\documentclass[a4paper]{article}

\usepackage[utf8]{inputenc}
\usepackage{lmodern}
\usepackage{amssymb,amsfonts,amsthm,amsmath,bbm,mathrsfs,aliascnt,mathtools}
\usepackage[english]{babel}

\usepackage{tikz}
\usetikzlibrary{decorations}

\usetikzlibrary{decorations.pathreplacing}

\tikzset{individu/.style={draw,thick}}
\usepackage{graphicx}
\usepackage[T1]{fontenc}
\usepackage{libertine}
\usepackage[libertine,liby,libaltvw,vvarbb]{newtxmath}

\theoremstyle{plain}
\newtheorem{theorem}{Theorem}[section]
\newtheorem{corollary}[theorem]{Corollary}
\newtheorem{lemma}[theorem]{Lemma}
\newtheorem{proposition}[theorem]{Proposition}

\theoremstyle{definition}

\newtheorem{assumption}[theorem]{Assumption}

\theoremstyle{remark}
\newtheorem{remark}[theorem]{Remark}

\numberwithin{equation}{section}

\usepackage{tikz}

\makeatletter \newcommand \listoftodos{\section*{Todo list} \@starttoc{tdo}}
\newcommand\l@todo[2]
{\par\noindent \textit{#2}, \parbox{10cm}{#1}\par} \makeatother

\newcommand{\N}{\mathbb{N}}

\newcommand{\R}{\mathbb{R}}
\newcommand{\C}{\mathbb{C}}

\newcommand{\calC}{\mathcal{C}}

\newcommand{\U}{\mathbb{U}}

\newcommand{\T}{\mathbb{T}}

\newcommand{\ind}[1]{\mathbf{1}_{\left\{#1\right\}}}
\newcommand{\indset}[1]{\mathbf{1}_{#1}}

\renewcommand{\tilde}[1]{\widetilde{#1}}
\renewcommand{\hat}[1]{\widehat{#1}}

\newcommand{\e}{\mathrm{e}}
\newcommand{\dd}{\mathrm{d}}

\newcommand{\SB}{^{\bullet}}
\newcommand{\svn}{_{\varnothing}}

\DeclareMathOperator{\E}{\mathbb{E}}
\newcommand{\Q}{\mathbb{Q}}
\renewcommand{\P}{\mathbb{P}}

\newcommand{\calF}{\mathcal{F}}

\newcommand{\calS}{\mathcal{S}}

\renewcommand{\epsilon}{\varepsilon}
\usepackage[colorlinks]{hyperref}

\usepackage{xcolor}

\title{Local times and excursions for self-similar Markov trees}
\author{Jean Bertoin\thanks{Institute of Mathematics, University of Zurich, Switzerland.}
\and Armand Riera\thanks{Sorbonne Universit\'e, LPSM, France.}
\and Alejandro Rosales-Ortiz\thanks{Institute of Mathematics, University of Zurich, Switzerland.} }
\date{ }

\begin{document}

\maketitle

\begin{abstract} This work builds upon the recent monograph \cite{ssMt} on self-similar Markov trees. A self-similar Markov tree is a random real tree equipped with a function from the tree to $[0,\infty)$ that we call the decoration.
Here, we construct local time measures $L(x,\dd t)$ at every  level $x>0$ of the decoration for a large class of self-similar Markov trees.  This enables us to mark at random a typical point in the tree at which the decoration is $x$. We identify the law of the decoration along  the branch from the root to this tagged point in terms of  a remarkable (positive) self-similar Markov process. We also show that after a proper normalization, $L(x,\dd t)$ converges as $x\to 0+$ to the harmonic measure $\upmu$ on the tree. Finally, we point out that using  a local time measure instead of the usual length measure $\uplambda$ to compute distances on the tree turn the latter into a continuous branching tree. This is relevant to analyze the excusions of the decoration away from a given level. Many results of the present work shall be compared with the recent ones in \cite{RRO1, RRO2} about local times and excursions of a Markov process indexed by  L\'evy tree. 
\end{abstract}


\medskip


\section{Introduction}
\label{sec:introduction}
The notion of local time  for a Markov process goes back to Blumenthal and Getoor \cite{BGLT} in the 60's.
Roughly speaking, this refers to a functional that measures the time that the Markov process spends at a given point of the state space.
Local times play an important role in stochastic analysis, notably in connection with potential theory and excursion theory.
Our purpose is to investigate an analog in the framework of self-similar Markov trees, a family of random processes that have been introduced recently  in  \cite{ssMt}.
These naturally  arise in  invariance principles for Galton--Watson processes with integer types, and as such, in a variety of scaling limit theorems for certain random structures. 

In short, a self-similar Markov tree, for which we often use the acronym ssMt, can be thought of as a (positive) self-similar Markov process in a branching setting. It
consists of a random function called the decoration,  $g: T\to [0,\infty)$,  which is defined on a  random rooted compact real tree $T$. 
Figure \ref{fig:ssMt} below may help to visualize a sample of a ssMt.
We can think of $T$
 as encoding the genealogy of a population that  starts from a single ancestor and becomes eventually extinct, where distances on $T$ are interpreted as time durations.
 Each individual in the population has a positive size that evolves  as time passes according to Markovian dynamics, begets children, and eventually dies. We think of $g(t)$ as the size of an individual at some time represented by $t\in T$. 
 In general, a parent begets infinitely many children; however, all but finitely many are minuscule with an extremely short life.
 As the name suggests, $(T,g)$ is self-similar and satisfies a version of the Markov property.

\begin{figure}[!h]
 \begin{center}
  \includegraphics[width=7cm]{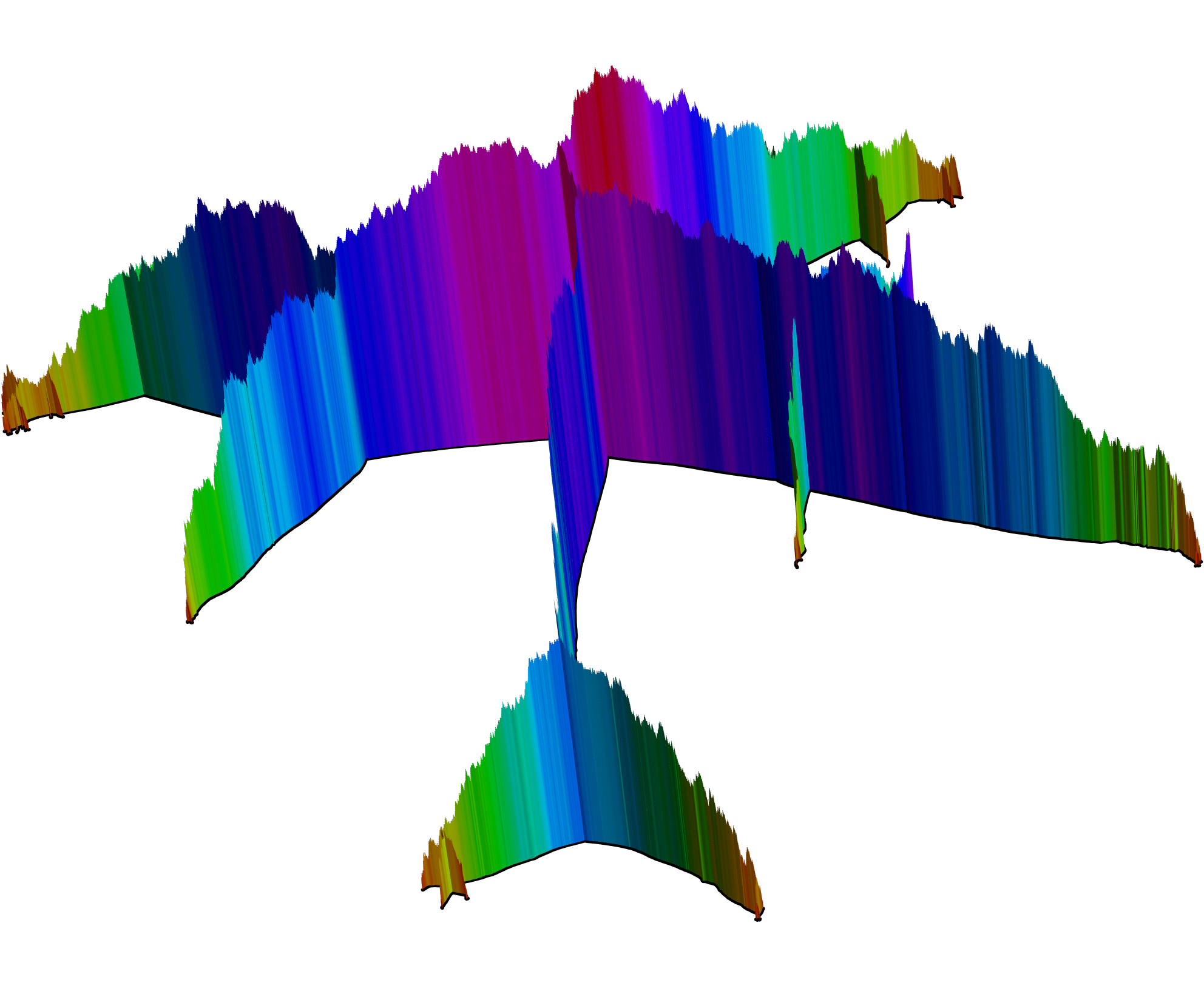}
 \caption{3D simulation of a self-similar Markov tree $(T,g)$: the tree $T$ is  embedded in a horizontal plane
  and  vertical coordinates indicate the values taken by  the decoration $g$ on $T$ .} \label{fig:ssMt}
 \end{center}
 \end{figure}

Our motivation for this work partly stems from the recent article  \cite{RRO1} by two of us, in which the local time at a regular point for a Markov process indexed by a L\'evy tree
has been studied in depth; see also \cite{RRO2} for applications to excursion theory. We stress that Markov processes on L\'evy trees and ssMt are fairly different stochastic processes. For the first, Markovian dynamics are superposed to a given L\'evy tree,
and as a consequence, there are no jumps at branching points. At the opposite, the decoration $g$ in a ssMt evolves simultaneously as the genealogical tree $T$ grows, the two are intrinsically coupled and $g$ is typically discontinuous at a branching point of $T$. 
Another striking difference can be observed by considering values taken at a fixed height in the tree.
The latter always form a null set (in the sense that for any $\epsilon>0$, only finitely many values are greater than $\epsilon$) for a ssMt, but this is usually not the case for a Markov process indexed by a L\'evy tree.
Despite their dissimilarities, both Markov processes indexed by L\'evy trees and ssMt are described by random functions on random continuous trees, and it is thus natural to investigate alike properties.

Any real tree is naturally endowed with a length (or Lebesgue) measure $\uplambda$, such that the $\uplambda$-measure of a segment in the tree is given by the distance between the two extremities of that segment. The starting point of this work is the observation that for any ssMt satisfying some fairly general assumptions that will be presented in due time,
one can construct a natural family of measures, $L(x, \dd t)$ for $x>0$,  on $T$ which are carried by the level sets of the decoration 
$$\mathcal{L}(x)=\{t\in T: g(t)=x\}.$$
Informally, they are given by 
$$L(x, \dd t) = x^{-\alpha}\delta(x-g(t)) \uplambda (\dd t),$$
where $\delta$ is the Dirac delta function and $\alpha>0$ the index of self-similarity (the normalization by  $x^{\alpha}$ is motivated by future convenience only). In particular, for any measurable function $\varphi: \R_+\times T \to \R_+$ with $\varphi(0,t)\equiv 0$, there is the occupation density formula,
\begin{equation}\label{E:occupdens}
\int_{T}  \uplambda(\dd t)  \varphi(g(t),t)= \int_0^{\infty} \dd x\,  x^{\alpha-1} \int_{T}  L(x,\dd t) \varphi (x,t) .
\end{equation}
As it is known that for any $\epsilon>0$, 
the upper level set of the decoration $\{t\in T: g(t)\geq \epsilon\}$ is carried by a finite number of segments of $T$ only, the construction of local times for ssMt in Section \ref{sec:construc} is essentially effortless --at least, once the general formalism for ssMt has been assimilated -- and definitely much more elementary than the construction of local times for a Markov process indexed by a Lévy tree in \cite{RRO1}.

Our main contributions can be summarized as follows. First, Theorem~\ref{T1} identifies the law of the decoration along the path from the root to a typical point of $\mathcal{L}(x)$, that is, a point sampled according to $L(x,\mathrm{d}t)$: we show that it evolves as a positive self-similar Markov process started from $1$ and conditioned to terminate at $x>0$. Second, under appropriate assumptions,
Theorem~\ref{theorem:convergence} shows that the family of local time measures $L(x,\mathrm{d}t)$, $x > 0$, converges to the harmonic measure $\mu(\mathrm{d}t)$ as $x\to 0+$. Third, Theorem~\ref{T:Lcbt} proves that the level set $\mathcal{L}(1)$, equipped with the distance induced by $L(1,\mathrm{d}t)$, has the law of a continuous branching tree denoted by $\tilde{\mathcal{L}}$. Finally, Theorem~\ref{T:EPPP} describes the structure of the 
excursions away from level $1$. Here, these are defined as the closures of the  connected components of $T\setminus \mathcal{L}(1)$ endowed with the corresponding restricted decoration.  Namely, we establish that, when the label of the root of the ssMt is $1$, then

\begin{itemize}
  \item excursions for which the root point is the only one with decoration~$1$ are in one-to-one correspondence with the leaves of $\tilde{\mathcal L}$;
  \item those that contain at least three points at which decoration is $1$ correspond to the branching points of $\tilde{\mathcal L}$;
  \item finally, excursions that contain exactly two points yield points of multiplicity two of $\tilde{\mathcal L}$.  
\end{itemize}

This result echoes the classical It\^o excursion theory.  An analogue in the setting of Markov processes indexed by L\'evy trees (where the continuous branching tree is replaced by a different L\'evy tree) can be found in~\cite{RRO2}. 
 The results presented in this work are in essence similar to the  corresponding ones in \cite{RRO1, RRO2}, but the techniques employed are of radically different nature due to the stark difference between the objects involved. For related earlier work, see \cite{ALG15}, which develops an excursion theory for Brownian motion indexed by the Brownian tree.  To streamline the exposition, we omit certain minor routine technical verifications.

The rest of this work is organized as follows. Section \ref{sec:nutshell} is a brief overview of the theory of ssMt with a focus on key elements needed in this work. Section \ref{sec:LP} provides some background on local times for L\'evy processes and potential densities.
We start entering the heart of the matter in Section \ref{sec:LTdeco} with the construction of local time measures on ssMt and the identification of the decoration from the root to a typical point of the level set $\mathcal{L}(x)$. Next, Section \ref{sec:convergenceH} is devoted to the convergence of the
local time measure $L(x, \dd t)$ towards the harmonic measure $\upmu(\dd t)$ as $x\to 0+$.  Finally, the branching structure of the level set and of the excursions away from the corresponding level is presented in Section~\ref{sec:branchstruc}.

\section{Self-similar Markov trees in a nutshell} 
\label{sec:nutshell}
The purpose of this section is to introduce some essential features of ssMt which will be needed for our goals, avoiding as much as possible technical issues when these are not important  for a first reading. Thus the reader should not expect to get here neither a precise construction nor rigorous statements (which has to be found in \cite[Chapters 1 and 2]{ssMt}), but rather an intuitive and sometimes only heuristic description of  key elements.

The celebrated Lamperti transformation establishes a bijective correspondence between real L\'evy processes and (positive) self-similar Markov processes. It also lies naturally at the heart of the construction of self-similar Markov trees. We briefly recall this transformation in a setting tailored for our purposes and refer to \cite[Chapter 13]{Kyp} for details; see also \cite{KypPar}.

Let $\xi$ denote a real L\'evy process that either drifts to $-\infty$, i.e. 
$\lim_{t\to \infty}\xi_t=-\infty$ almost-surely, 
or is killed at an exponentially distributed time $\zeta$ (we agree that $\zeta=\infty$ when there is no killing). By convention, if $\zeta<\infty$, we set $\xi_{\zeta+t}=-\infty$ for every $t\geq 0$. We write 
$(\sigma^2, \mathrm a, \Lambda)$ for  its characteristic triplet, where $\sigma^2\geq 0$ is the Gaussian coefficient, $\mathrm a\in \R$ the drift coefficient,
and $\Lambda$ the L\'evy measure, with the convention that the killing rate is  incorporated to the L\'evy measure as the mass at $-\infty$. 
We  also write  $\psi$ for the L\'evy exponent, so that
$$E\left(\exp(\gamma \xi_t)\indset {t<\zeta} \right) = \exp\big(t\psi(\gamma)\big), \qquad t\geq 0,$$
for any $\gamma\in \C$ such that the expectation in the left-hand side is well-defined.
This is in particular the case when $\gamma=i \theta$ is purely imaginary, and $-\psi(0)\geq 0$ is the killing rate.
In this setting, the L\'evy-Khintchine formula reads
$$\psi(\gamma) =
\frac{1}{2}\sigma^2 \gamma^2+ {\mathrm a} \gamma + \int_{\R^*}   \left( \mathrm{e}^{\gamma y} -1-  \gamma y \mathbf{1}_{|y|\leq 1} \right)\Lambda ( \dd y) - \Lambda(\{-\infty\}).$$
 
 We next fix some $\alpha>0$ to serve as index of self-similarity.
The process $X=(X_t)_{0\leq t < z}$ with a.s. finite lifetime 
$$z=\int_0^\zeta \exp(\alpha \xi_s) \dd s$$ and defined for $t<z$ by  
\begin{equation} \label{E:Lamperti}
X_t=\exp(\xi_s), \qquad \text{where  } t=\int_0^s \exp(\alpha \xi_r) \dd r,
\end{equation}
is then a Markov process, which is further self-similar with index $\alpha$. By convention, we set $X_z=0$.  Plainly, its law  is determined by that of the L\'evy process and the index of self-similarity, and 
thus more precisely, either by the quadruplet $(\sigma^2, \mathrm a, \Lambda; \alpha)$, or, equivalently, by the pair $(\psi; \alpha)$.

Similarly, the law of a self-similar Markov tree is  characterized  by a quadruplet $(\sigma^2, \mathrm a, \boldsymbol{\Lambda}; \alpha)$, where  now 
 $\boldsymbol{\Lambda}$ is a measure on the product space $[-\infty, \infty) \times \calS_1$
 and $\calS_1$ denotes the space of non-increasing sequences $\mathbf y=(y_n)_{n\geq 1}$ with $y_n\in[-\infty, \infty)$ and $
 \lim_{n\to \infty} y_n=-\infty$. We write $\Lambda_0$ for the image of $\boldsymbol{\Lambda}$ by the first projection 
  and always assume that $\Lambda_0$ is a L\'evy measure on $[-\infty,\infty)$ (recall that the possible mass at $-\infty$ is interpreted as a killing rate). We refer to  $\boldsymbol{\Lambda}$ as the generalized L\'evy measure. 
 Roughly speaking, the triplet $(\sigma^2, \mathrm a, \boldsymbol{\Lambda})$ enables us to define a branching L\'evy process, much in the same way as for the L\'evy-It\^{o} construction of a L\'evy process. Slightly less informally, this involves Poisson random measures on $\R_+\times [-\infty, \infty) \times \calS_1$ with intensity $\dd t \boldsymbol{\Lambda} (\dd y_0, \dd \mathbf y)$. An atom at $(t, y_0, \mathbf y)$ indicates a birth event; the first component $t$ represents the time at which this event occurs, the second component $y_0$ corresponds to 
  the size of a jump of the parent at time $t$, whereas the terms $y_n$, $n\geq1$, of the sequence  $\mathbf y$ determine the relative locations of the children at birth.  We refer to \cite{BM} for a detailed construction.
  Roughly speaking, the self-similar Markov tree  is then obtained by applying the Lamperti transformation for each branch (or individual)  of the branching L\'evy process.

 Let us explain a bit more precisely, though still informally, the construction of the tree $T$ and its decoration $g$. Recall that a ssMt can be thought as depicting a population evolving in continuous time. 
 We standardly  label every individual by a finite sequence of positive integers and use the Ulam tree $\U=\bigcup_{n\geq 0}\N^n$
 to encode the genealogy, where by convention $\N^0=\{\varnothing\}$.  
 This means that the $j$-th child\footnote{This means that all the children of an individual are ranked according to some algorithm. For formalism convenience, we allow fictitious children, in the sense that some labels in $\U$ may not correspond to any real individual in the population. The descent of a fictitious individual is of course entirely fictitious, and those fictitious individuals can be fully ignored in the construction.}
 of an individual at the $n$-th generation, say $u=(u_1, \ldots, u_n)$, is $uj=(u_1, \ldots, u_n, j)$.
 Each individual $u\in \U$ is identified with an oriented segment  $S_u$ of $T$; the length $|S_u|$ of $S_u$ is the lifespan of $u$. 
 The left-extremity of  the ancestral segment $S\svn$ may be thought of as the birth of the ancestor of the population; it is denoted by $\uprho$ and serves as root.
In turn, for every $u\in \U$ and $j\geq 1$,   the left-extremity of  
 the segment  $S_{uj}$
 corresponds to the birth of the $j$-th child and  is glued on the parent segment $S_u$ at the  location corresponding to the age of the parent at the birth event. Starting from the ancestor $v=\varnothing$, we thus construct $T$ by gluing segments, generation after generation.  The decoration  on a  segment $S_u$ is denoted by $f_u$; it specifies the size of the individual $u$ as a function of its age. Its dynamics are those of a self-similar Markov process with self-similarity index $\alpha$, and can be constructed
 by applying the Lamperti transformation to a L\'evy process with characteristics $(\sigma^2, \mathrm a, \Lambda_0)$. The offspring of an individual $u\in \U$
 is encoded by a point process $\eta_u$ on $S_u\times (0,\infty)$ called the reproduction process, where the first coordinate of the atoms represent the ages of $u$ at the times when it begets children, and the second coordinate represent the size at birth (i.e. the initial value of the decoration) of a child.
 
 We refer the reader to \cite[Chapters 1 and 2]{ssMt}, and especially Section 2.2 there for a precise and rigorous construction. This requires a  crucial sub-criticality condition, \cite[Assumption 2.8]{ssMt}, which we now recall. We first define the cumulant $\kappa$ as the function taking values in $(-\infty, \infty]$ associated with the characteristic quadruplet $(\sigma^2, a, \boldsymbol{\Lambda}; \alpha)$ as follows:
 \begin{equation} \label{E:kappa} \kappa(\gamma)
:= \psi(\gamma)+ \int_{ \calS_1}  \boldsymbol{\Lambda}_1(\dd  \mathbf(y_n)_{n\geq 1} ) \sum_{i=1}^{\infty} \e^{\gamma y_n},
\end{equation}
where $\psi$ stands for the L\'evy exponent with characteristics $(\sigma^2, \mathrm a, \Lambda_0)$ and $ \boldsymbol{\Lambda}_1$ for the image of 
the generalized L\'evy measure $ \boldsymbol{\Lambda}$ by the second projection on $\calS_1$. 
 \begin{assumption} \label{A:subcrit} 
We suppose that there 
  exists $\gamma_0>0$ such that $$\kappa(\gamma_0)<0.$$
 \end{assumption}
In short, Assumption \ref{A:subcrit} ensures that almost-surely, the family  $(\sup f_u)_{u\in \U}$ of the maximal size reached by each individual, is null, in the sense that for any $\epsilon>0$,
\begin{equation} \label{E:deconull}
\text{the set }\{u\in\U :\sup f_u>\epsilon\} \text{ is finite},
\end{equation}
and also that
\begin{equation}\label{sum:finite}
\sum_{n=0}^{\infty} \max_{|u|=n} |S_u|< \infty.
\end{equation}
In words, the series of the maximal lifespans of individuals at generations $n=0,1, \ldots$ converges, which  ensure the compactness of $T$. 
We stress that nonetheless, the  length measure $\uplambda$ is in general nowhere locally finite on $T$.
 
 We shall assume for simplicity that the decoration at the root is normalized to $1$, as by self-similarity, this entails no loss of generality. 
 We use the notation $\P$ for the distribution of the family indexed by $u\in\U$ of the decoration-reproduction processes $(f_u, \eta_u)$ on each  segment $S_u$.
 Turning our attention to decorated trees themselves, we have to address a conceptual difficulty.
 Loosely speaking, we are only concerned with the structure of decorated trees rather than specific realizations such as those resulting from gluing; 
 this leads us to identify isomorphic decorated trees and work on the Polish space  
$ \mathbb{T}$  of isomorphic classes of decorated trees, see \cite[Section 1.4]{ssMt}.   We endowed it with $\Q$, the image of $\P$
by the map which first constructs a decorated tree from a family of decoration-reproduction processes and then sends a decorated tree to its equivalence class up to isomorphism
in $ \mathbb{T}$.
 In particular,   $\P$ determines $\Q$, but in general different characteristic quadruplets can yield the same law $\Q$, because 
  one cannot always fully recover the population model (i.e. the decoration-reproduction processes) from the decorated tree alone. 
  Nonetheless, many variables considered in \cite{ssMt} or in the present text are intrinsic, in the sense that even though they can be conveniently defined at first in terms of 
  specific realization of a decorated tree, they are invariant by isomorphisms and can henceforth be treated as variables on $ \mathbb{T}$.
  
 The expectation under $\P$ or $\Q$ is denoted indifferently by $\E$. Occasionally, we will have to deal with the rescaled version  of the ssMt that has initial decoration $g(\uprho)=x>0$ and will then use the notation $\P_x$ and $\Q_x$. 
 By rescaling, we mean that 
 the law of the rescaled decorated tree $(x^{\alpha}T,xg)$ under $\Q$ is $\Q_x$, where the notation $x^{\alpha}T$ indicates that distances in $T$ are dilated by a factor $x^{\alpha}$;
  see \cite[Proposition 2.9]{ssMt}. 
 
\section{Some background on local times of L\'evy processes}
\label{sec:LP}
Real valued  -- possibly killed-- L\'evy processes and their local times play a crucial part in this work.
In this section, we briefly recall  some elements that will be needed, and for this it is convenient to introduce the canonic formalism.

We work on the space $\mathbb{D}$ of rcll paths $\xi: \R_+\to \R\cup\{-\infty\}$,  where $-\infty$ serves as a cemetery point at which paths are absorbed.
We write  
$$\zeta:=\inf\{t\geq 0: \xi_t=-\infty\},$$
 for the lifetime of $\xi$ and agree that any function $f$ on $\R$ is extended to $[-\infty, \infty)$ with $f(-\infty)=0$,
so $f(\xi_t)=0$ when $t\geq \zeta$. We endow as usual  $\mathbb{D}$ with the Skorokhod topology
and 
denote the natural filtration by $(\calF_t)_{t\geq 0}$.

We consider a probability measure $P$ on  $\mathbb{D}$  that describes the law of the L\'evy process started from $0$,
and more generally, we denote the law of the version of the L\'evy  process started from $x\in \R$ by $P_x$. In other words, $P_x$ is the distribution of $x+\xi$ under $P$. 
As in the preceding section, we write $\psi$ for the L\'evy exponent and
assume that  either   the L\'evy process drifts to $-\infty$ or is killed at a constant rate.

\subsection{Local time measures and level sets}\label{sec:ltmls}
Throughout this work, we request  the existence of local times 
$$\big(\ell(x,t): x\in \R\text{ and }t\geq 0\big),$$
in the sense that the occupation density formula
\begin{equation}\label{E:occden}
\int_0^{t} f(\xi_s) \dd s = \int_0^{t\wedge \zeta} f(\xi_s) \dd s = \int_{-\infty}^{\infty} f(x) \ell(x,t) \dd x
\end{equation}
holds  for any  measurable function $f: \R\to \R_+$ and $t\geq 0$, $P$-a.s.
In our setting, this is known to be equivalent to the integral condition for the L\'evy exponent $\psi$,  
\begin{equation}\label{E:abs}
\int_{-\infty}^{\infty} \Re\left( \frac{1}{1-\psi(i\theta )}\right)\dd \theta<\infty,
\end{equation}
where the notation $\Re$ refers to the real part of a complex number.
This assertion can be seen from the combination of  \cite[Theorem V.1]{LP} and \cite[Remark 37.7]{Sato}.

We also request  $0$ to be regular for itself, in the sense that
\begin{equation}\label{E:reg}
P(H_0=0)=1,
\end{equation}
where 
$$H_y:=\inf\big\{t>0: \xi_t=y\big\} $$
denotes the first hitting time of a level $y\in \R$. 
We recall from \cite[Proposition V.2]{LP} that we can then choose for each level  $y$ a  version of the local time process $ \ell(y,\cdot)$ 
such that, for any $t'>0$ and $x\in \R$, we have
\begin{equation} \label{E:cvl2}
\lim_{\epsilon \to 0+} \sup_{0\leq t \leq t'} \left | \frac{1}{2\epsilon} \int_0^t \indset{|\xi_s-y|<\epsilon} \dd s-\ell(y,t)\right | =0, \qquad \text{in }L^2(P_x).
\end{equation}
Then the local time process $t\mapsto \ell(y,t)$ is a continuous additive functional, and \eqref{E:cvl2}  suggests to view the local time (Stieltjes) measure $\ell(y,\dd t)$ as the analog of the uniform measure on the level set $\{0\leq t < \zeta: \xi_t=y\}$. In this direction, it is easily checked that the support of the local time measure $\ell(y, \dd t)$  coincides with 
the level set\footnote{Although the sample paths of a L\'evy process are in general discontinuous, its level set at $y$ is closed $P_x$-a.s., since
the probability that a jump of $\xi$ either starts or ends at $y$ is zero. See e.g. the proof of the forthcoming Lemma \ref{L:wellknown}.}  at $y$, $P_x$-a.s.

Plainly,  the local time process reaches its terminal value
$\ell(y,\infty)<\infty$ at its last-passage time
$$K_y=\sup\big\{t\geq 0: \xi_t=y\big\},$$
which is finite $P_x$-a.s., since either the killing rate is strictly positive or the L\'evy process is transient. 
Therefore the occupation formula \eqref{E:occden} extends to $t=\infty$. 
Further, $\ell(0,\infty)$ has an exponential distribution under $P$, and more generally, $\ell(y,\infty)$  follows the same exponential law conditionally on the event $H_y<\infty$. 
We write
$$v(y)= E(\ell(y,\infty)), \qquad y\in \R,$$
and refer to the function $v$ as the potential density\footnote{This is also called the Green function.}, since one has
\begin{equation}\label{E:abspot}
\int_0^{\infty} E\left(f(\xi_t)\right) \dd t = \int_{-\infty}^{\infty} f(y) v(y) \dd y,
\end{equation}
for any measurable function $f\geq 0$. The potential density  is a bounded, continuous and everywhere positive function, see \cite[Theorem II.19]{LP} or \cite[Section 43]{Sato}. It determines hitting probabilities,
\begin{equation}\label{E:hit}
P_x(H_y<\infty)=
P_x(H_y<\zeta)=v(y-x)/v(0);
\end{equation}
in particular $v$ reaches its unique maximum at $0$. We further infer from \eqref{E:hit} that $\lim_{|y|\to \infty}v(y)=0$ when  killing occurs, i.e. $-\psi(0)>0$. In that case, the potential density is also integrable and can be identified as the continuous function with Fourier transform
\begin{equation} \label{E:Fourier}
\int_{-\infty}^{\infty} \e^{i \theta x } v(x) \dd x= -\frac{1}{\psi(i\theta)}, \qquad \theta \in \R.
\end{equation}

 It will be useful later on to check that the approximation \eqref{E:cvl2} also holds $P_x$-almost-surely 
for any $x\in \R$, at least provided that the limit in the left-hand side is taken along some sequence of positive real numbers tending to $0$ fast enough.
We stress that we want to use the same sequence for all initial values $x\in \R$, 
the weaker statement where the sequence may depend on $x$ is indeed plain from \eqref{E:cvl2}. 

\begin{lemma}\label{L:conslt} There exists a sequence $(\epsilon_n)_{n\geq 1}$ of positive real numbers with $\epsilon_n\to0$, such that for any $x,y\in \R$ and $t'>0$,
$$\lim_{n\to \infty} \sup_{0\leq t \leq t'} \left | \frac{1}{2\epsilon_n} \int_0^t \indset{|\xi_s-y|<\epsilon_n} \dd s-\ell(y,t)\right | =0, \qquad P_x\text{-a.s.}$$
As a consequence, $P_x$-a.s.,  the sequence of random measures
$$\frac{1}{2\epsilon_n}  \indset{|\xi_t-y|<\epsilon_n} \dd t$$
converges towards the Stieltjes measure $\ell(y, \dd t)$ in the sense of weak convergence of finite measures  on $\R_+$.
\end{lemma}
\begin{proof} There is no loss of generality to assume that $y=0$ and that the killing rate is positive. Consider first the case when $x=0$.
It follows from \eqref{E:cvl2}
that there is a sequence $(\epsilon'_n)_{n\geq 0}$ of positive real numbers tending to $0$ fast enough so that the convergence in the claim holds  $P$-a.s.
Now let the L\'evy process start from an arbitrary $x\in \R$ and apply the Markov property at the first hitting time of $0$.
Since $\ell(0,H_0)=0$, 
all that we need now is to check that, uniformly in $x$, the time spent by $\xi$ before time $H_0$ in an interval $(-\epsilon, \epsilon)$ is asymptotically small compared to $\epsilon\ll 1$ with high probability under $P_x$.

So consider the L\'evy process started from $x$ and killed when hitting $0$, $(\xi_t)_{0\leq t < H_0}$. By \eqref{E:hit}, the latter is a Markov process with potential density  $v^\dagger$ given by
$$v^\dagger(x,y')=v(y'-x)-\frac{v(-x)}{v(0)} v(y'), \qquad y'\in \R.$$
Since $v$ is uniformly continuous (recall that we assumed that $-\psi(0)>0$ so that $v\in \mathcal C^0$), we can extract from the sequence $(\epsilon'_n)_{n\geq 0}$ a sub-sequence $(\epsilon_n)_{n\geq 0}$ such that
$$v^\dagger(x,y')\leq 2^{-n} \qquad \text{ for all $x\in \R$,  $n\geq 0$, and $|y'|<\epsilon_n$. }$$
In particular, we have
$$E_x\left(\frac{1}{2\epsilon_n} \int_0^{H_0} \indset{|\xi_t|<\epsilon_n} \dd t \right) = \frac{1}{2\epsilon_n} \int_{-\epsilon_n}^{\epsilon_n} v^\dagger(x,y')\dd y' \leq 2^{-n},$$
and a standard combination of the Borel-Cantelli lemma and the Markov inequality entails that 
$$\lim_{n\to \infty} \frac{1}{2\epsilon_n} \int_0^{H_0} \indset{|\xi_t|<\epsilon_n} \dd t= 0, \qquad P_x\text{-a.s.}$$
This completes the proof of the first claim in the statement. The second claim is then clear.
\end{proof}

\subsection{Conditioning on  the terminal value} \label{sec:condter}
We suppose throughout this section that the killing rate is positive, so that the lifetime $\zeta$ 
is exponentially distributed with parameter $\Lambda(\{-\infty\})=-\psi(0)>0$ under $P_x$  for any $x\in \R$. 
It is readily seen from the Fubini theorem that  the potential density then also naturally appears in the distribution of the terminal value, i.e. the location of the L\'evy process
at the end of its lifetime. Specifically one has
\begin{equation}\label{E:ter}P_x( \xi_{\zeta-}\in \dd y) = -\psi(0) v(y-x)\dd y, \qquad x,y\in \R.
\end{equation}

Our main purpose in this section is to rigorously define, and then present some basic properties of, a version of the conditional distribution of the L\'evy process given its terminal value. Such questions have been addressed in depth by Fourati \cite{Fourati} (see also the references therein) in a more general setting than ours,  using an approach  \textit{\`a la }Doob combined with so-called Kuznetsov measures. We shall address first conditioning on  a finite terminal value $\xi_{\zeta } = y$ for $y\in \R$, and we shall then turn our attention to the case ``$y = - \infty$'' by studying the weak limit as $y\to -\infty$.

\subsubsection{Finite terminal value}
For any fixed 
 $y\in \R$,  the function $x\mapsto v(y-x)$ is excessive for the L\'evy process and the process $v(y-\xi_{\cdot})$ is a $P_x$-supermartingale for all $x\in \R$.
This incites us to define a probability measure $P_{x,y}$ on $\mathbb{D}$ 
such that  for every $t\geq 0$ and every event $A\in \calF_t$, we have:
\begin{equation}\label{E:Doob}
P_{x,y}\left(A \cap \{t<\zeta\}\right) = \frac{1}{v(y-x)} E_{x}\left( v(y-\xi_t)\indset A\right).
\end{equation}
To make the connection with conditioning on the terminal value, 
observe from \eqref{E:ter}, the Markov property, and the continuity of the potential density, that the conditional probability
$$P_{x}\left(  A \cap \{t<\zeta\} \mid  \xi_{\zeta-}\in(y-\epsilon,y+\epsilon)\right )$$
converges as $\epsilon\to 0+$ towards the right-hand side of \eqref{E:Doob}.

There is actually a technical difficulty with this approach, namely one has to check that  \eqref{E:Doob}  yields indeed a probability measure $P_{x,y}$ on the space $\mathbb{D}$ of rcll paths absorbed at $-\infty$ (the issue is to establish the left-continuity at lifetime) and that $P_{x,y}$-a.s., the terminal value is  $\xi_{\zeta-}=y$.
Nonetheless, this problem can be circumvented  in our setting with the following well-known direct construction in terms of killing at a last-passage time.

\begin{lemma}\label{L:lastpt}  Under the conditional law $P_x(\cdot \mid H_y<\infty)$, the distribution of the rcll process obtained by killing $\xi$ at its last-passage at $y$,
$(\xi_t)_{0\leq t < K_y}$, satisfies  \eqref{E:Doob}.
\end{lemma}
\begin{remark}Plainly for $x=y$, we have $P_x(\cdot \mid H_x<\infty)=P_x$, thanks to \eqref{E:reg}. \end{remark}
\begin{proof} It suffices to observe that $K_y>0$ if and only if $H_y<\infty$, and from \eqref{E:hit} and the Markov property, that 
$$P_x(t<K_y \mid \mathcal F_t) = v(y-\xi_t)/v(0).$$
We infer 
 that for any $x,y\in \R$, any $t\geq 0$ and any event $A\in \calF_t$, we have 
$$P_x\left( A \cap \{t<K_y\}\mid H_y<\infty \right) = \frac{1}{v(y-x)} E_{x}\left( v(y-\xi_t) \indset A \right).
$$
\end{proof}
\noindent Since   finite-dimensional distributions determine probability measures on $\mathbb{D}$, Lemma~\ref{L:lastpt} now makes clear the existence and uniqueness  of the law  $P_{x,y}$ on  $\mathbb{D}$  satisfying \eqref{E:Doob}. 
We can then check that  the family $(P_{x,y})_{y\in \R}$  provides the regular disintegration of the law $P_x$ with respect to the terminal value.

\begin{proposition}\label{P:disint} 
For every  starting point $x\in \R$, the map $y\mapsto P_{x,y}$ with values in the space
of probability measures on the Skorokhod space $\mathbb{D}$ endowed with the topology of weak convergence, is continuous, and one has the disintegration identity
\begin{equation}\label{equation:desintegration}
    P_x=-\psi(0)\int_{-\infty}^{\infty}  v(y-x) P_{x,y} \dd y.
\end{equation}
\end{proposition}

\begin{proof} Consider any sequence $(y_n)_{n\geq 0}$ of real numbers that tends to $y$.
 It is plain from \eqref{E:Doob} and the continuity of $v$ that $P_{x,y_n}$ converges to $P_{x,y}$ as $n\to \infty$, in the sense of finite-dimensional distributions.
On the other hand, the representation of $P_{x,y}$ as the law of the L\'evy process killed at its last passage time at $y$ in Lemma \ref{L:lastpt} enables us to check
the general criterion for tightness in \cite[Theorem VI.3.21]{JS} for the sequence $(P_{x,y_n})_{n\geq 0}$. Indeed, it allows us to bound from above the probabilities there 
computed under the laws $P_{x,y_n}$ by those computed under $P_x$ up to a constant factor. We conclude that as $n\to \infty$, $P_{x,y_n}$ converges to $P_{x,y}$
weakly on $\mathbb{D}$. The disintegration in the statement now makes sense and is plain from \eqref{E:ter} and \eqref{E:Doob}.
\end{proof}

We next point at an elementary identity involving the Esscher transform, which will be important for the present study.
\begin{lemma} \label{L:expofam} Take any  $\beta\in \R$ with  $\psi(\beta)<0$ and write
 $\psi^{(\beta)}(\cdot)=\psi(\beta+\cdot)$ for the shifted L\'evy exponent. Then the following assertions hold:
\begin{itemize} 
\item[(i)]  $\psi^{(\beta)}$ is the L\'evy exponent of another killed L\'evy process,  whose law  $P^{(\beta)}$ 
is equivalent to $P$ with density
\begin{equation} \label{E:tilted}\frac{\dd P^{(\beta)}}{\dd P}= \frac{\psi(\beta)}{\psi(0)}\exp(\beta \xi_{\zeta-}).
\end{equation}
Furthermore, the potential density $v^{(\beta)}$ of this $\beta$-tilted L\'evy process is
$$v^{(\beta)}(x)=\e^{\beta x} v(x), \quad  x \in \mathbb{R}.$$

\item[(ii)] For every $x,y\in \R$, the conditional laws $P_{x,y}$ and $P^{(\beta)}_{x,y}$ are identical.
\end{itemize}
\end{lemma}
\begin{proof} (i) It follows from \eqref{E:ter} that the positive random variable $\frac{\psi(\beta)}{\psi(0)}\exp(\beta \xi_{\zeta-})$ has expectation $1$ under $P$. Thus 
\eqref{E:tilted} defines a probability measure $P^{(\beta)}$ on the space $\mathbb{D}$ of rcll paths, which may be referred to as the $\beta$-tilted law.
Take any $t\geq 0$ and write $\exp(\beta \xi_{\zeta-})= \exp(\beta \xi_t) \exp(\beta (\xi_{\zeta-}-\xi_t))$ on the event $t<\zeta$.
We deduce from the Markov property of $\xi$ that for
every event $A\in \calF_t$, there is the identity
$$P^{(\beta)}\left( A \cap \{t<\zeta\}\right )=  E\left( \exp(\beta \xi_t)\indset A\right).$$
We then readily infer 
 (e.g. by computing the characteristic functions of increments) that $P^{(\beta)}$ is indeed the law of another L\'evy process with L\'evy exponent $\psi^{(\beta)}$.
 The stated formula for its potential density is immediate.

(ii) The identity $P_{x,y}=P^{(\beta)}_{x,y}$ follows readily from (i),  \eqref{E:Doob} and  \eqref{E:tilted}. \end{proof}

We conclude this section with two more results about the conditioned L\'evy process.
The first is a time-reversal identity, which was also discussed by Fourati \cite{Fourati}.
The opposite $-\xi$ of a L\'evy process is again a L\'evy process; we write 
$\hat P_x$ for its law started from $x$, i.e. for the law of $-\xi$ under $P_{-x}$. 
The following claim is readily seen from the elementary duality lemma \cite[Section II.1]{LP} combined with Proposition \ref{P:disint}. 

\begin{corollary}\label{C:dual} For any $x,y\in \R$, the law under $P_{x,y}$ of the time-reversed process,
$$\xi_{(\zeta-t)-}, \quad 0\leq t < \zeta,$$ 
is $\hat P_{y,x}$.
\end{corollary}

Finally, an easy consequence of \eqref{E:cvl2} is that the process $v(y-\xi_t)+\ell(y,t)$ for $t\geq 0$ is a $P_x$-martingale for all $x,y\in \R$.
Its terminal value $\ell(y, \infty)= \ell(y,\zeta)$ has expectation $v(y-x)$, and more precisely,  is exponentially distributed with mean $v(0)$ under $P_x(\cdot \mid H_y<\infty)$. 
Roughly speaking, this entails that conditioning on the terminal value  can also be interpreted as a killing according to the local time measure.
Here is the precise connection. 

\begin{lemma} \label{L:ltbiased}  For every $x,y\in \R$,  the identity
$$E_{x,y}\left(A_{\zeta}\right) = \frac{1}{v(y-x)} E_x\left( \int_0^{\zeta} A_t \ell(y, \dd t )  \right) $$
holds for any bounded $(\calF_t)$-predictable process $(A_t)_{t\geq 0}$.
\end{lemma}

\begin{proof}
Since  the process $v(y-\xi_t)+\ell(y,t)$ for $t\geq 0$ is a $P_x$-martingale, we have for any  bounded $(\calF_t)$-progressively measurable process $(B_t)_{t\geq 0}$ that
$$ E_x\big( B_t (\ell(y,\zeta)-\ell(y,t))\big)=E_x( B_t v(y-\xi_t) ).$$ 
Lemma \ref{L:lastpt} then enables us to rewrite the right-hand side as
$$ v(y-x) E_{x,y} \left( B_t\indset{t<\zeta} \right).$$
Integrating over $t\geq 0$ and applying Fubini's theorem, we arrive at
$$ E_x\left( \int_0^{\zeta} \ell(y, \dd r) \int_0^r \dd t B_t \right)= v(y-x) E_{x,y} \left( \int_0^{\zeta} \dd t B_t \right).$$
This is our claim when the predictable process can be expressed in the form
$$A_t = \int_0^t \dd r B_r, \qquad t\geq 0,$$
for some  bounded progressively measurable process $(B_t)_{t\geq 0}$. The general case follows by a standard argument.
\end{proof}
\subsubsection{Infinite terminal value}

The purpose of this section is to study the limiting behavior of the measures $P_{x,y}$ as $y\to -\infty$. This requires an additional assumption of Cramer's type  on the L\'evy exponent $\psi$, namely that:

\begin{equation}\label{E:CramerLP}
\text{there exists $\varrho <0$   such that $\psi(\varrho)=0$.}
\end{equation}
For the rest of the section we suppose that \eqref{E:CramerLP} is satisfied. 
Note from convexity of $\psi$ that  the right-derivative at $\varrho$ is negative,  $\psi'(\varrho)<0$.
The exponent $\psi$  shifted by $\varrho$ is still a L\'evy exponent, and the first item of  Lemma \ref{L:expofam} admits the following finite-horizon analog :
 
\begin{lemma} \label{L:expofam:gamma:-} 
The function  $\psi^{(\varrho)}$ is the L\'evy exponent of another  L\'evy process which has no killing and drifts towards $-\infty$. Its  law  $P^{(\varrho)}$ verifies for every $t>0$
\begin{equation} \label{E:tilted:2}  P^{(\varrho)}_{|\mathcal{F}_t}= \exp(\varrho \xi_{t}) P_{|\mathcal{F}_t}.
\end{equation}
Furthermore, the potential density $v^{(\varrho)}$ of this $\varrho$-tilted L\'evy process is bounded and given by 
\begin{equation}\label{equation:asymtotvCramer}
    v^{(\varrho)}(x)=\e^{ \varrho x} v(x), \quad  x \in \mathbb{R}.
\end{equation}
\end{lemma}
The proof  is analogous to that for Lemma \ref{L:expofam} and  we skip the details. The absence of killing stems from  $\psi^{(\varrho)}(0) = 0$, 
and the drift towards $- \infty$ from  $\psi^{(\varrho)\prime}(0)=\psi'(\varrho)<0$. The rest of the section is devoted to establishing the convergence of  $P_{x,y}$ towards $P_x^{(\varrho)}$ as $y \to - \infty$. \begin{proposition}\label{convergence:P:x:y}
For every $x\in \mathbb{R}$, the following  convergence holds weakly   with respect to the Skorokhod topology: 
\begin{equation} \label{equation:convergencePxy}
    \lim \limits_{y\to  -\infty} P_{x,y}= P_x^{(\varrho)}. 
\end{equation}
As a consequence, we have also,  
$$    \lim \limits_{y\to  -\infty} P_{x}(  \, \cdot \,  | \xi_{\zeta-} \leq y)= P_x^{(\varrho)}. $$
\end{proposition}

 The proof of Proposition \ref{convergence:P:x:y} relies on the representation \eqref{E:Doob} of the measures $P_{x,y}$ paired with the following asymptotic estimate of the potential density $v$.

\begin{lemma} \label{lemma:asymv} We have:
$$\lim_{y\to -\infty} \exp(\varrho y) v(y)=-1/\psi^{\prime}(\varrho). $$
\end{lemma}

\begin{proof}
Recalling \eqref{equation:asymtotvCramer}, the claim can be rephrased as 
    \begin{equation} \label{equation:convergenceDensity}
        \lim_{y \to - \infty} v^{(\varrho)}(y) = -1/\psi'(\varrho). 
    \end{equation}
    Let us first explain why this result should not come as a surprise. The Lévy process with exponent $\psi^{(\varrho)}$ is non-lattice, has no killing and has finite mean $\psi'(\varrho)<0$.  The renewal theorem  \cite[Theorem I.21]{LP} yields that for every $h > 0$, 
    \begin{equation}\label{renewal:lim}
        \lim_{x \to - \infty} \int_{x}^{x+h} v^{(\varrho)}(y) \mathrm{d} y = - h /\psi'(\varrho),  
    \end{equation}
 in agreement with \eqref{equation:convergenceDensity}. 
 
 On the other hand, we observe from \eqref{E:hit} and the strong Markov property at $H_x$ that 
 $$\frac{v^{(\varrho)}(x)v^{(\varrho)}(y-x)}{v^{(\varrho)}(0)^2} = P^{(\varrho)}_0(H_x<\infty) P^{(\varrho)}_x(H_y<\infty)\leq 
P^{(\varrho)}_0(H_y<\zeta)=v^{(\varrho)}(y)/v^{(\varrho)}(0).$$
We now readily see that \eqref{equation:convergenceDensity} follows from \eqref{renewal:lim} and the fact that $v^{(\varrho)}(y-x)$ converges to $v^{(\varrho)}(0)$ as $y-x\to 0$. 
\end{proof}

We  have all the ingredients to establish the convergence of $P_{x,y}$ as $y \to -\infty$.

 \begin{proof}[Proof of Proposition \ref{convergence:P:x:y}]
Fix $x\in \mathbb{R}$ and  $t > 0$. Recall from \eqref{E:Doob}  that for every measurable  functional $F\geq 0$ on the space $\mathbb{D}([0,t])$ of possibly absorbed rcll paths indexed by $[0,t]$,   we have 
     \begin{equation}\label{equation:step1}
         E_{x,y}\big(F(\xi_{s}, 0 \leq s \leq t )1_{\zeta > t}\big) 
         = \frac{1}{v(y-x)} E_x\big( F(\xi_{s }, 0 \leq s \leq t ) v(y-\xi_t) \big).  
     \end{equation}
     for   $y \in \mathbb{R}$. 
It readily follows from Lemma \ref{L:expofam:gamma:-} that 
\begin{equation*}
\sup_{(y,z)\in \R_{-}\times \R}   \exp(  -\varrho z  ) \frac{v(y-z)}{v(y)} <\infty.
\end{equation*}
 Since $E_x(\exp(\varrho \xi_t))=\e^{\varrho x}<\infty$,  dominated convergence applies and Lemma \ref{lemma:asymv} yields
\begin{align*}
\lim \limits_{y\to -\infty}E_{x,y}\big(F(\xi_{s}, 0 \leq s \leq t )1_{\zeta > t}\big) &=  E_x\big( F(\xi_{s }, 0 \leq s \leq t )  \exp(\varrho (\xi_t-x) ) \big)\\
&=  E_x^{(\varrho)}\big( F(\xi_{s }, 0 \leq s \leq t ) \big),
\end{align*}
where the second equality stems from Lemma \ref{L:expofam:gamma:-}. In particular, since under $P_x^{(\varrho)}$ the process $\xi$ does not have killing,  from taking $F=1$ we infer that $P_{x,y}(\zeta>t)\to1$ as $y \to - \infty$.  It then follows that
$$
\lim \limits_{y\to -\infty}E_{x,y}\big(F(\xi_{s}, 0 \leq s \leq t ) \big) = E_x^{(\varrho)}\big( F(\xi_{s }, 0 \leq s \leq t ) \big),
$$
for every every measurable  functional $F\geq 0$ on $\mathbb{D}([0,t])$,  and since $t > 0$ is arbitrary the desired convergence \eqref{equation:convergencePxy} follows. The second statement of the proposition is now a straightforward consequence of the disintegration \eqref{equation:desintegration} and \eqref{equation:convergencePxy}.  
\end{proof}
\section{Local time measures on decoration level sets} 
\label{sec:LTdeco}

Throughout the rest of this work, we consider a 
characteristic  quadruplet  $(\sigma^2, \mathrm a, \boldsymbol{\Lambda}; \alpha)$ that satisfies Assumption \ref{A:subcrit}. 
Recall from Section \ref{sec:nutshell} that  $\P$ denotes the law of the family of decoration-reproduction processes induced by this quadruplet when the decoration at the root $\uprho$  is $g(\uprho)=1$, and  $\Q$ that of the self-similar Markov tree $(T,g)$. We  further request that the L\'evy process $\xi$ with exponent $\psi$ and characteristics $(\sigma^2, \mathrm a, \Lambda_0)$ satisfies  \eqref{E:abs} and \eqref{E:reg}. 
\begin{figure}[!h]
 \begin{center}
  \includegraphics[width=5cm,height=4cm]{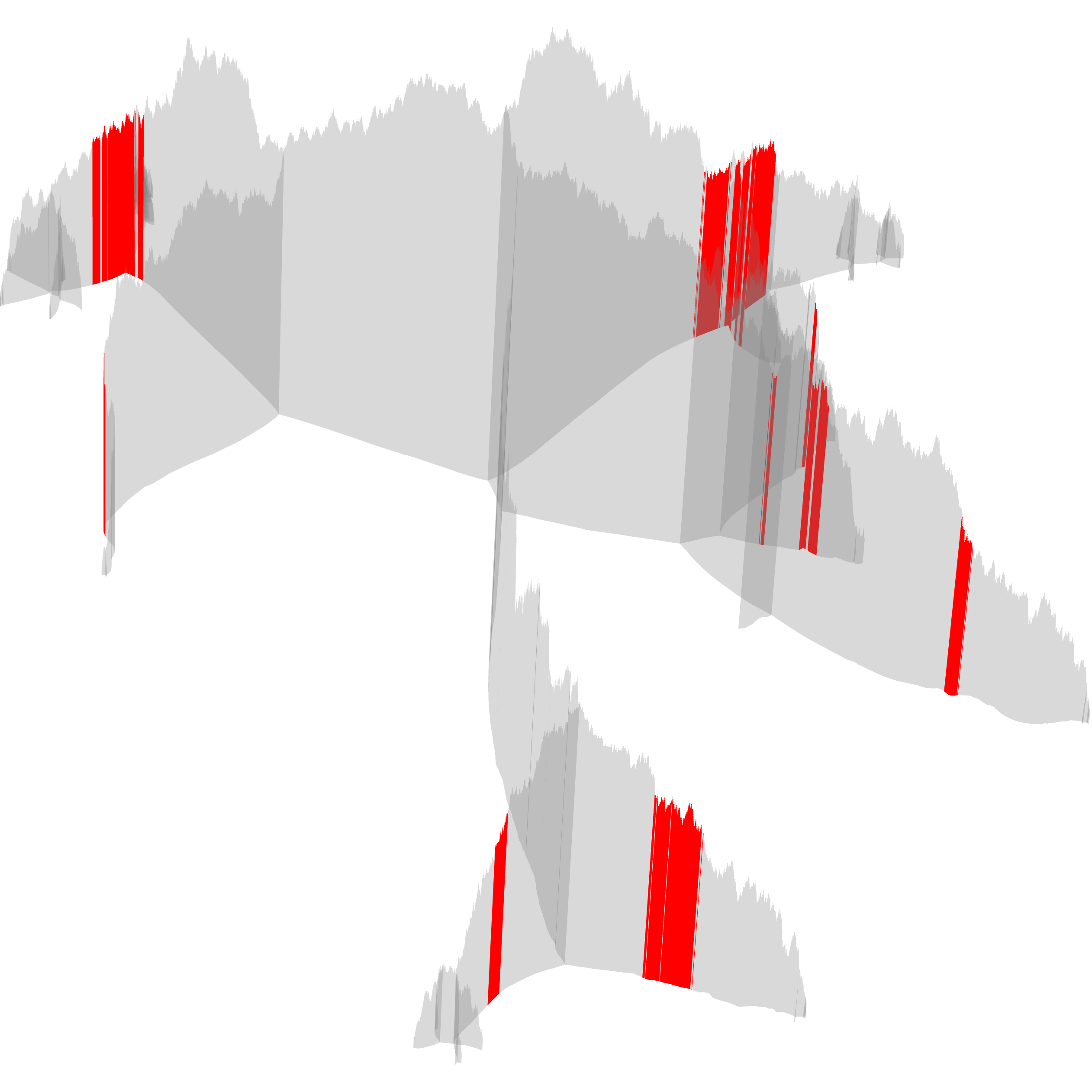}
 \caption{3D simulation of a self-similar Markov tree $(T,g)$ started from $1$. In red are represented the points where the decoration takes a fixed value $x$ (here $x=1/2$). } \label{fig:ssMt2}
 \end{center}
 \end{figure}
\subsection{Construction of the local time measure} \label{sec:construc}
We know from Section \ref{sec:ltmls} that $\xi$ possesses local time measures $\ell(y, \dd s)$ at all levels $y\in \R$; this is readily transferred to self-similar Markov processes
by the Lamperti transformation \eqref{E:Lamperti}. Specifically, we define the local time measure $L(x, \dd t)$ at level $x=\e^{y}$ on $[0,z)$ as the image of
$\ell(y, \dd s)$ on $[0, \zeta)$ by the time-change
$$s\mapsto t = \int_0^s \exp(\alpha \xi_r) \dd r, \qquad s\in [0, \zeta).$$ 
In other words, the repartition function  $L(x,t) =L(x, [0,t])$ for $0\leq t \leq z$ is simply given by
$$L(x,t)=\ell(y,s).$$
The image of the Lebesgue measure $\dd s$ on $[0, \zeta)$ by the time-change is $X_t^{-\alpha} \dd t$ on $[0,z)$, 
 so the occupation density formula for the self-similar Markov process $X$ reads
 \begin{align*}\int_0^t f(X_r) \dd r &=\int_0^s f(\e^{\xi_u}) \e^{\alpha \xi_u}\dd u\\
 &=\int_{-\infty}^{\infty} f(\e^y) \e^{\alpha y} \ell(y, s) \dd y\\
 &= \int_{0}^{\infty} f(x) L(x,t) x^{\alpha-1} \dd x
 \end{align*}
 for any measurable function $f: (0,\infty)\to \R_+$ and $t<z$. 

 We also make a couple of simple observations in this setting that will be useful in the sequel. First, 
 since for any $y, y_0\in \R$, $\ell(y-y_0, \cdot)$ is the local time process at level $y$ of the shifted L\'evy process $\xi+y_0$,
 the local times of a self-similar Markov process also inherit  the scaling property. Specifically, if $X'_t= cX_{c^{-\alpha}t}$ is the rescaled version of $X$ started from $c=\e^{y_0}$, then the local time at level $x$ for $X'$ is 
 $$L'(x,t)= L(c^{-1}x, c^{-\alpha}t).$$
  Second,  we have $L(x,z)=\ell(\log x, \infty)$, and the expectation of this variable when the self-similar Markov process starts from $x'>0$, that is 
 when the L\'evy process starts from $\log x'$, is 
 \begin{equation} \label{E:meanltssM}
 E(L(x,z) \mid X_0=x') = v(\log(x/x')),
 \end{equation} where $v\in\calC^0$ denotes the potential density for the L\'evy exponent $\psi$.

 It should now be plain how to construct local times measures on ssMt. Recall that the decoration $g$ on $T$ stems from trajectories of self-similar Markov processes 
 with characteristics $(\psi; \alpha)$ on a family of segments $(S_u)_{u\in \U}$.  We  can define a local time measure $L(x, \dd t)$
 on each segment $S_u$, and the summation over all segments yields the local time measure $L(x, \dd t)$ on $T$. We stress from \eqref{E:deconull} that only finitely many segments carry a non-zero local time measure at level $x>0$, and hence $L(x, \dd t)$ is a finite measure.  
 
 The reader should not be confused by the fact that we use the same notation $L(x, \dd t)$ for the local time measure at level $x$, both for the self-similar Markov process $X$ and  for the self-similar Markov tree $(T,g)$, because these are in essence the same functional.
 Similarly, using $t$ for a typical time in $[0,z)$ or a typical point in the tree $T$ is intentional\footnote{"La math\'ematique est l'art de donner le m\^eme nom \`a des choses diff\'erentes." (Henri Poincar\'e, in \textit{Science et M\'ethode}, 1908)}.
The cautious reader may however worry that our construction of local times 
 is not intrinsic but rather seems to depend on the family of decoration-reproduction processes instead  of $(T,g)$ only. This possible concern should be appeased by 
  the following observation; recall that $\uplambda$ denotes the length measure on $T$.
  
  \begin{proposition} \label{P:conslt} Let $(\epsilon_n)_{n\geq 1}$ be a sequence of positive real numbers as in Lemma \ref{L:conslt}. For every $x>0$, the 
  limit below holds $\P$-a.s. in the sense of weak convergence on the space of finite measures on $T$:
 $$
\lim_{n\to \infty} \frac{1}{2\epsilon_n}  \indset{|\log(g(t)/x)|<\epsilon_n} \uplambda(\dd t) 
= x^{\alpha} L(x, \dd t).
$$
  Furthermore, the occupation density formula \eqref{E:occupdens} holds, $\P$-a.s. 
  \end{proposition}
  \begin{proof} Indeed, an
  immediate consequence of Lemma \ref{L:conslt}, the Lamperti transformation, and the gluing construction, is 
$$
\lim_{n\to \infty} \frac{1}{2\epsilon_n}  g(t)^{-\alpha} \indset{|\log(g(t)/x)|<\epsilon_n} \uplambda(\dd t) 
= L(x, \dd t), \qquad \P\text{-a.s.}
$$
The claims in the statement derive easily. 
  \end{proof}

Clearly, the scaling property of ssMt propagates to their local time measures. In particular, 
for any $c>0$, the joint distribution under $\Q$ of the rescaled ssMt together with its total local times
 $$\left( c^{\alpha} T, cg, (L(x/c, T))_{x>0}\right)$$
 is the same as that of  $\left( T, g, (L(x, T))_{x>0}\right)$ under $\Q_c$. 
 Note also that, by construction,  the local time measure has no atoms and that its support coincides with the level set, 
$$\mathrm{Supp }\, L(x, \dd t) = \{t\in T: g(t)=x\}, \qquad \Q\text{-a.s.}$$

\subsection{Total local times and their mean values}
We turn our attention to  the total local times $L(x,T)$ at levels $x$, which form a family of $\P$-a.s. finite random variables.
A special case of the occupation density formula \eqref{E:occupdens} for functions $\varphi$ depending on $x$ only reads
\begin{equation}\label{E:occdens}
\int_{T} f(g(t)) \uplambda(\dd t) = \int_0^{\infty} f(x) x^{\alpha-1} L(x,T) \dd x
\end{equation}
for any measurable function $f: \R_+\to \R_+$ with $f(0)=0$.

We next point out that the variables $L(x,T)$ are integrable and their expectations can be computed as follows. 
Recall the notation \eqref{E:kappa} for the cumulant function $\kappa$,   Assumption \ref{A:subcrit}  and the definition of $\gamma_0$ therein. 

\begin{proposition}\label{P:exlt} The following assertions hold.
\begin{enumerate}
\item[(i)]
The shifted cumulant, 
$$\kappa^{(\gamma_0)}(\cdot)=\kappa(\gamma_0+\cdot)$$ is a L\'evy exponent with killing for which \eqref{E:occden} and \eqref{E:reg} are satisfied.
We write $\mathrm{w}^{(\gamma_0)}\in \calC^0$ for its potential density.
\item[(ii)] For every $x>0$, we have
$$\E\big( L(x,T)\big) = x^{-\gamma_0}\mathrm{w}^{(\gamma_0)}(\log x).$$
\end{enumerate}
\end{proposition} 
We stress that, thanks to Lemma \ref{L:expofam}(i) and as it should be expected,  the expression in (ii) actually does not depend on the choice of the real number $\gamma_0$ for which Assumption \ref{A:subcrit} is verified. 
\begin{proof}  (i) The claim is a refined version of  \cite[Lemma 5.5]{ssMt}. 
Since $\psi(\gamma_0)< \kappa(\gamma_0)<0$,  it is immediately checked (much as in the proof of Lemma \ref{L:expofam}(i)),  that the function
$$\phi(\cdot)=\psi(\gamma_0+\cdot)+ \kappa(\gamma_0)-\psi(\gamma_0)$$ 
is the exponent of a killed L\'evy process,  whose distribution is equivalent to the original one on any finite time-interval. In particular \eqref{E:occden} and \eqref{E:reg} are satisfied for $\phi$. On the other hand, we see from the very definition \eqref{E:kappa} that $\kappa^{(\gamma_0)}-\phi$ is the L\'evy exponent of a compound Poisson process 
(i.e. remains bounded on the purely imaginary line). As a consequence,  $\kappa^{(\gamma_0)}$  fulfils \eqref{E:occden} and \eqref{E:reg} as well.

 (ii) Without loss of generality, we may assume that $x'=1$. It is convenient to use the notation $\chi(u)$ for the value taken by the decoration at the origin (i.e. the left-extremity) of the segment
 $S_u$ for $u\in \U$, just as in  \cite[Chapter 2]{ssMt}. From the construction of the local time at level $x$ and \eqref{E:meanltssM}, we get
 the identity
 $$\E\left( L(x,T)\right) = \E \Big( \sum_{u\in \U} v\big(\log(x/\chi(u))\big) \Big).$$
 
 Consider the point processes on $\R$ induced by the logarithm of the  values of the decoration at origins of segments,
 $$\sum_{u\in \U} \updelta_{\log \chi(u)}.$$
 It is known that the generating function of its intensity is given by
 \begin{equation}\label{eq:transformée:chi:gamma}
 \E\left( \sum_{u\in \U} \chi(u)^{\gamma}\right) =\psi(\gamma)/\kappa(\gamma), \qquad\text{  whenever }\kappa\big(\Re(\gamma)\big)<0.
 \end{equation}
 See Lemma 2.7  and the proof of Proposition 2.12 in \cite{ssMt} (the calculation there is done when $\gamma$ is a real number, but remains valid more generally when $\gamma$  is complex). 
 
So, if we introduce the finite measure $\mu$ on $\R$ such that
$$\int_{\R} f(y) \mu(\dd y) = \E\left( \sum_{u\in \U} \chi(u)^{\gamma_0} f(\log(\chi(u))\right), $$
then the latter  has Fourier transform
$$\int_{\R} \e^{i \theta y} \mu(\dd y) = \frac{\psi(\gamma_0+i \theta)}{\kappa(\gamma_0+i\theta)} \ , \qquad \theta \in \R.$$
We  set $v^{(\gamma_0)}(y)= \e^{\gamma_0y} v(y) $, and we arrive at the expression in terms of a convolution
$$\E\left( L(x,T)\right) = x^{-\gamma_0} \cdot \left(v^{(\gamma_0)} \ast \mu\right)  (\log x).$$
 This quantity depends continuously on  $x>0$, since $v^{(\gamma_0)}\in \calC^0$ and $\mu$ is a finite measure. By \eqref{E:Fourier} for $\psi$ and \eqref{eq:transformée:chi:gamma}, the convolution $v^{(\gamma_0)} \ast \mu$  has Fourier transform $-1/\kappa^{(\gamma_0)}$, 
 and we can complete the identification using again \eqref{E:Fourier}, now for the L\'evy exponent $\kappa^{(\gamma_0)}$. \end{proof}
 
 As a quick verification of Proposition \ref{P:exlt}(ii), observe that, thanks to the occupation density formula \eqref{E:occdens},  the so-called weighted length measure $\uplambda^{\gamma_0}$ defined in \cite[Proposition 2.12]{ssMt}  by 
\begin{equation}\label{E:weightedlength}
\uplambda^{\gamma_0}(\dd t) = g(t)^{\gamma_0-\alpha}(t)\uplambda(\dd t), \qquad t\in T,
\end{equation}
has total mass
$$\uplambda^{\gamma_0}(T) = \int_0^{\infty}   x^{\gamma_0-1}   L(x, T) \dd x.$$
We know from \cite[Proposition 2.12]{ssMt} that
\begin{equation}\label{E:exlambda}
\E\left(\uplambda^{\gamma_0}(T)\right)= -1/\kappa(\gamma_0).
\end{equation}
The calculation in the proof of Proposition \ref{P:exlt}(ii) shows that this quantity indeed equals
$$\int_0^{\infty} x^{\gamma_0-1} \E(L(x,T)) \dd x.$$
We stress that this is however insufficient to recover the formula of Proposition \ref{P:exlt}(ii), even in the situation where $\gamma_0$ may be arbitrarily chosen in some open interval, since Proposition \ref{P:exlt}(ii) holds for all  --and not just almost all-- $x>0$.
\\

\begin{remark} One important feature of the function $x\mapsto L(x,T)$ is that, up to a multiplicative factor of $x^{\alpha-1}$, it is the occupation density  of the decoration. The study of occupation densities of random processes has been conducted in a wide range of settings. For instance, occupation densities have  been extensively studied for superprocesses with Brownian displacements, and these results can be further extended to tree‐indexed Brownian motion.  Concretely, Sugitani \cite{Sugitani} first proved that super‐Brownian motion with branching mechanism $\psi(\lambda)=2\lambda^2$ admits an occupation density, and this was later extended to general $\alpha$\nobreakdash-stable branching mechanisms in \cite{MytPer}.  Thanks to the connection between superprocesses and Lévy snakes \cite{DLG}, one can derive the existence of an occupation density for Brownian motion indexed by an $\alpha$-stable Lévy tree.  In the special case when $\alpha=2$, this yields the celebrated Integrated super Brownian excursion  $f_{\mathrm{ISE}} = (f_{\mathrm{ISE}}(x), x\in\R)$, which  has played a  fundamental role in scaling‐limit theorems of discrete trees, and appears in interacting‐particle systems and statistical‐physics models \cite{DS,HaS,HoS}. The study of $f_{\mathrm{ISE}}$ remains an active field of research. For instance, recent work of Le Gall and Perkins shows that $f_{\mathrm{ISE}}$ satisfies a stochastic differential equation \cite{LeGallPerkins2025}. It is worth mentioning that the naïve  analogue of the total mass $L(x,T)$ in the setting of Brownian motion indexed by a stable tree is almost surely degenerate (taking only the values $0$ or $\infty$), so the construction of  local time measures is not as straightforward as in this work, and requires additional care.  This task is carried over in \cite{RRO1}, where a local time measure for general  Markov processes indexed by  Lévy trees is obtained by an approximation procedure. 
\end{remark}

\subsection{Decoration from the root to a typical level point}\label{sec:decora:spine}
A key result about ssMt is the so-called many-to-one formula based on the weighted length measure 
$\uplambda^{\gamma_0}$,
see \cite[Corollary 5.8]{ssMt}. Let us present informally a weaker version that focusses on the spine and discards the so-called dangling subtrees attached to it.
Imagine that we mark a point $\uprho^{\bullet}$ in $T$ at random proportionally to $\uplambda^{\gamma_0}$. Then the decoration along the spine, that is the semi-open segment $\llbracket \uprho, \uprho^{\bullet}\llbracket$ from the root to this marked point, has the distribution of a self-similar Markov process with characteristics
$(\kappa^{(\gamma_0)}; \alpha)$. 
One of the main results of this work can be thought of as a version of the latter when the marked point is rather sampled uniformly at random in the level set $\{t\in T: g(t)=x'\}$. We now prepare some material to give a rigorous statement. 

Fix $x>0$. The normalized  local time measure yields  a natural probability measure $\Q^x$ on the space $\T^{\bullet}$ of decorated real trees with a marked point, $(T,g,\uprho')$, such that for every nonnegative functional $F$ on this space,
\begin{equation} \label{equation:spinalL}
    \E^x(F(T,g,\uprho')) = \frac{x^{\gamma_0}}{\mathrm{w}^{(\gamma_0)}(\log x)}\E\left( \int_T F(T,g,t) L(x, \dd t)\right).
\end{equation}
We are interested in the process
$$Y= g_{\mid \llbracket \uprho, \uprho'\rrbracket }$$
defined by the restriction of 
the decoration $g$  to the segment $\llbracket \uprho, \uprho'\rrbracket $ from the root $\uprho$ to the distinguished point $\uprho'$.
We  conveniently view  $Y$ as a process indexed by nonnegative times $Y=(Y_t)_{0\leq t \leq z'}$
by identifying $t\geq 0$ with the point of the segment $\llbracket \uprho, \uprho'\rrbracket $ at distance $t$ from the root.

We consider the self-similar Markov process with characteristics $(\kappa^{(\gamma_0)}; \alpha)$ and for every $x>0$, 
we write  $Q_{1,x}$ for the conditional version of its law starting from 1 and given that its terminal value is $x>0$. More precisely, $Q_{1,x}$ is the law of the process obtained by applying the Lamperti transformation to the conditional version of a L\'evy process with exponent $\kappa^{(\gamma_0)}$ started from $0$ and conditioned --in the sense of Section \ref{sec:condter}-- to terminate at $y=\log x$. We stress that,  thanks to Lemma \ref{L:expofam},  the law $Q_{1,x}$  does not depend on choice of  $\gamma_0$ in Assumption \ref{A:subcrit}, even though the unconditional law $Q=Q_1$ does.

\begin{theorem} \label{T1} 
Under $\Q^x$, $Y$ has the law $Q_{1,x}$. 
\end{theorem}
\begin{figure}[!h]
 \begin{center}
  \includegraphics[width=5cm,height=5cm]{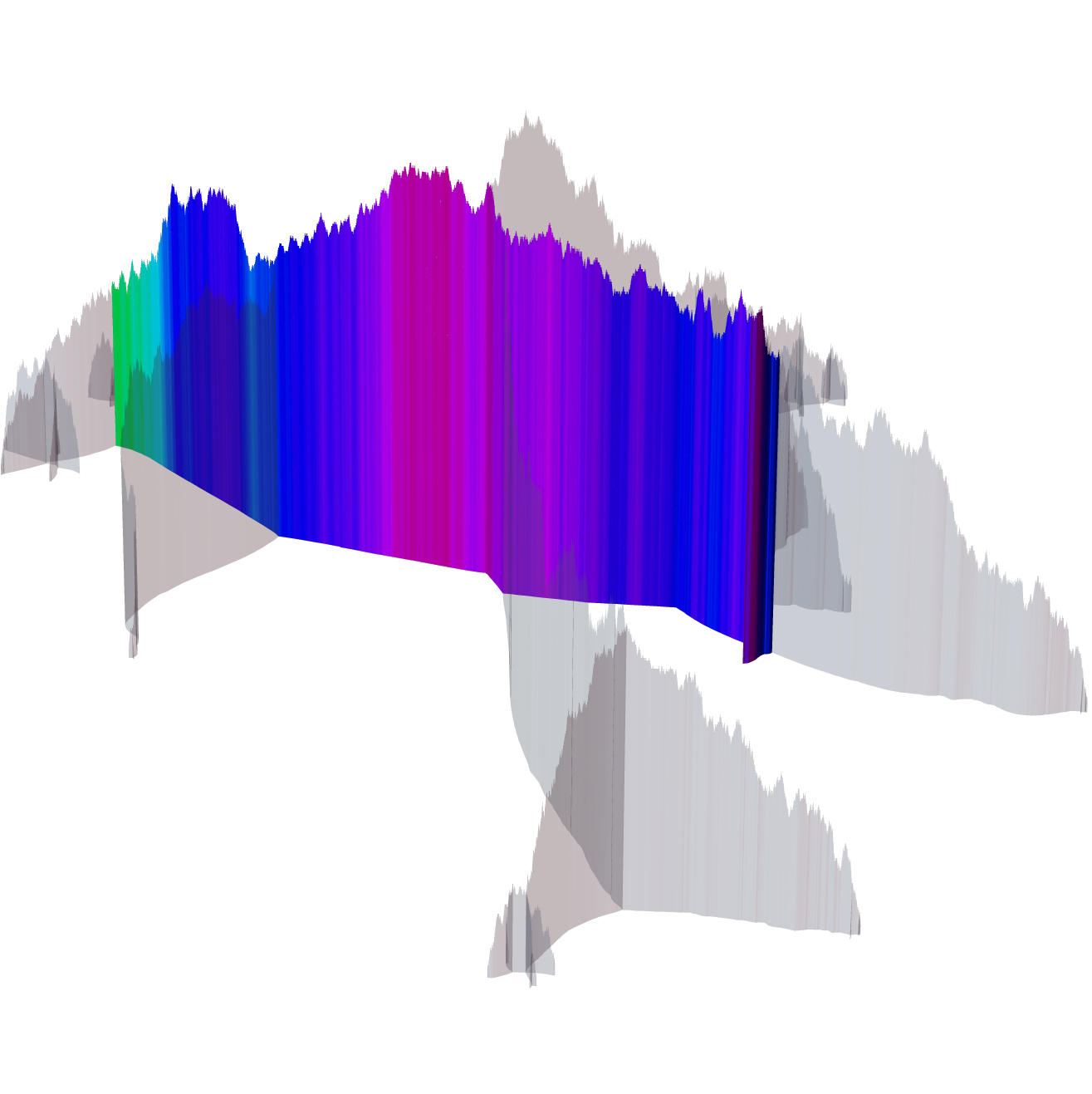}
 \caption{Illustration of Theorem \ref{T1}. The spine to a typical $L(x,\mathrm{d}t)$ point and the associated decoration is represented in color, and the rest of the decorated tree in gray. Here the decoration starts at the root at $1$ and we take $x=1/2$. The distribution of the decoration along this spine is distributed $Q_{1,x}$.} \label{fig:ssMt3}
 \end{center}
 \end{figure}
It may be also noteworthy to compare Theorem \ref{T1} with Lemma \ref{L:ltbiased}; the latter is indeed a special case of the former when $\kappa=\psi$, that is,  in absence of branching.

\begin{proof} In short, the idea is to deduce the claim from the many-to-one formula when the marked point is sampled according to  
$\uplambda^{\gamma_0}$. Roughly speaking, we re-weight $\uplambda^{\gamma_0}$ and make it close to the local time measure  $L(x, \dd t)$. The effect of 
the re-weighting on the spine amounts to forcing the decoration at the right-extremity of the spine (i.e. the marked point) to be close to $x$.

Recall Proposition \ref{P:conslt} and the sequence $(\epsilon_n)_{n\geq 1}$ of positive real numbers there. This incites us to use the density $\frac{1}{2\epsilon_n}  g(t)^{-\gamma_0} \indset{|\log(g(t)/x)|<\epsilon_n}$ to re-weight the measure  $\uplambda^{\gamma_0}$. 
Next, consider any continuous function with compact support $\varphi: \R_+\times \R_+\to \R_+$, and  define a function $F(T,g,\cdot)$ on $T$ by
 $$F(T,g,t) = \int_{\llbracket \uprho, t\rrbracket} \varphi(g(s),s) \dd s,\qquad t\in T,$$
 where we recall that $\llbracket \uprho, t\rrbracket$ denotes the segment from the root to $t$, and by a slight abuse of notation, we identify 
 a point $s$ in this segment with its distance to the root and write $\dd s $ for the length measure. So $F(T,g,t)$ is a smooth functional of the restriction of the decoration to the 
 segment $\llbracket \uprho, t\rrbracket$, and 
 the function $F(T,g,\cdot)$
  is continuous on $T$, $\Q_x$-a.s. Therefore we have from Proposition \ref{P:conslt} that
 $$
 \lim_{n\to \infty} \frac{1}{2\epsilon_n} \int_T F(T,g,t) g(t)^{-\gamma_0} \indset{|\log(g(t)/x)|<\epsilon_n} \uplambda^{\gamma_0}(\dd t)
 =  \int_T F(T,g,t) L(x, \dd t).
$$
We claim that the limit above also holds in $L^1(\Q)$. Indeed, it follows from Proposition~\ref{P:exlt}(ii) and the occupation density formula that 
\begin{align*} \frac{1}{2\epsilon_n} \E\left( \int_T g(t)^{-\alpha} \indset{|\log(g(t)/x)|<\epsilon_n} \uplambda(\dd t)\right) & =  \frac{1}{2\epsilon_n} \int_0^{\infty} y^{-1} \indset{|\log(y/x)|<\epsilon_n}\E(L(y, T)) \dd y\\
&\rightarrow  \E(L(x,T)) \qquad \text{ as }n\to \infty.
\end{align*}
So by the Scheff\'e lemma, the convergence
 $$
 \lim_{n\to \infty} \frac{1}{2\epsilon_n} \int_T g(t)^{-\alpha} \indset{|\log(g(t)/x)|<\epsilon_n} \uplambda(\dd t)
 =  L(x,T)
$$
also takes place in $L^1(\Q)$, which suffices for our claim. 

We now compute 
 $$
 \frac{1}{2\epsilon_n} \E\left(   \int_T F(T,g,t) g(t)^{-\gamma_0} \indset{|\log(g(t)/x)|<\epsilon_n} \uplambda^{\gamma_0}(\dd t) \right)
$$
using the many-to-one formula \cite[Corollary 5.8]{ssMt} for the weighted length measure $\uplambda^{\gamma_0}$.  Recalling \eqref{E:exlambda}, 
we get
$$-\frac{ 1}{2\epsilon_n \kappa(\gamma_0)}Q_1\left( \left( \int_0^{z} \varphi(X_s,s) \dd s \right) \cdot  X_{z-}^{-\gamma_0} \indset{|\log(X_{z-}/x)|<\epsilon_n} \right),$$
where  $Q_1$ denotes the law of the self-similar Markov process  $(X_s)_{0\leq s < z}$ with characteristics $(\kappa^{(\gamma_0)}; \alpha)$ and started from $1$. 

Applying the Lamperti transformation and using Proposition \ref{P:disint} for the L\'evy process with exponent $\kappa^{(\gamma_0)}$, 
we now get that
\begin{align*}
& \lim_{n\to \infty} \frac{1}{2\epsilon_n}Q_1\left( \left( \int_0^{z} \varphi(X_s,s) \dd s \right) \cdot  X_{z-}^{-\gamma_0} \indset{|\log(X_{z-}/x)|<\epsilon_n} \right) \\
&= -\kappa(\gamma_0)x^{-\gamma_0} \mathrm{w}^{(\gamma_0)}(\log x) Q_{1,x}\left( \int_0^{z} \varphi(X_s,s) \dd s \right),
\end{align*}
where $Q_{1,x}$ is the law of the image by the Lamperti transformation of a L\'evy process with exponent $\kappa^{(\gamma_0)}$ started from $0$ and conditioned 
(in the sense of Section \ref{sec:condter}) to terminate at $\log x$.

Putting the pieces together, we have established the identity
$$ \E\left(   \int_T F(T,g,t) L(x,\dd t) \right)=x^{-\gamma_0} \mathrm{w}^{(\gamma_0)}(\log x) Q_{1,x}\left( \int_0^{z} \varphi(X_s,s) \dd s \right),
$$
Recalling Proposition \ref{P:exlt} (ii), this entails our claim.
\end{proof}

As in the discussion following the proof of Proposition \ref{P:exlt} (ii), let us point out that if one could establish the continuity of the map $x \longmapsto \E^x\big(F(T,g,\uprho')\big)$, $x>0, $
then Theorem \ref{T1} would follow directly by a disintegration argument. Let us briefly outline this approach. Fix a function $F(T,g,\cdot)$ and a test function $\varphi$ as in the previous proof. It is enough to verify that
\begin{equation}\label{eq:theo:4.3}
\E^x\big(F(T,g,\uprho')\big)
=
Q_{1,x}\Big( \int_0^{z} \varphi(X_s,s)\, \dd s \Big).
\end{equation}
To this end, introduce an auxiliary bounded continuous function $h:\mathbb{R}\to\mathbb{R}_+$. Applying the many-to-one formula for the weighted length measure $\uplambda^{\gamma_0}$ (see \cite[Corollary 5.8]{ssMt}), we obtain
\begin{align}\label{eq:1:spinal:g:gamma}
\E\Big(\int_T F(T,g,t)\, h(g(t))\, g(t)^{\gamma_0-\alpha}\, \uplambda(\dd t)\Big)
=
Q_1\Big(\frac{h(X_{z-})}{-\kappa(\gamma_0)}\int_0^{z} \varphi(X_s,s)\, \dd s\Big).
\end{align}
We may now invoke Proposition \ref{P:disint} to rewrite the right-hand side as
$$
\int_0^{\infty}
\frac{\mathrm{w}^{(\gamma_0)}(\log x)}{x}\,
h(x)\,
Q_{1,x}\Big(\int_0^{z} \varphi(X_s,s)\, \dd s\Big)\,
\dd x.
$$
Moreover, by the occupation density formula, the left-hand side of \eqref{eq:1:spinal:g:gamma} can be expressed as
$$
\int_0^{\infty}
\frac{\mathrm{w}^{(\gamma_0)}(\log x)}{x}\,
h(x)\,
\E^x\big(F(T,g,\rho^\prime)\big)\,
\dd x.
$$ 
Comparing these two expressions, we deduce that \eqref{eq:theo:4.3} holds for Lebesgue-almost every $x>0$. Finally, Proposition \ref{P:disint} ensures that the right-hand side of \eqref{eq:theo:4.3} is a continuous function of $x>0$ and  in order to conclude, it suffices to establish the continuity of the map $x\mapsto \E^x\big(F(T,g,\uprho')\big)$. This is immediate when the local times of the underlying Lévy process $(\ell(x,t): x\in\mathbb{R}\text{ and } t\geq 0)$ are jointly continuous, by combining Proposition \ref{P:exlt}(ii) with a dominated convergence argument. However, under our weaker assumptions,  the desired continuity does not follow directly. Proving it would require an argument akin to that used in the proof of Proposition \ref{P:exlt}, but substantially more involved than the one presented in the above proof. 
\\
\\

We next present a pair of immediate  consequence of Theorem \ref{T1} involving time-reversal, which are closely related to the so-called Riesz-Bogdan-{\.Z}ak transformation for self-similar Markov processes; see \cite{Alili, Kypsurvey}.
For every $x>0$, 
we write $\hat Q_x$ for the law of the self-similar Markov process with characteristics $(\hat \kappa^{(\gamma_0)}; \alpha)$ started from $x$,
where $\hat \kappa^{(\gamma_0)}(\cdot)= \kappa(\gamma_0-\cdot)$ is the dual L\'evy exponent, 
and then  $\hat Q_{x,1}$ for its conditional version given that the terminal value is $1$. 

\begin{corollary} \label{C:reversed} The following two assertions hold under $\Q^x$:
\begin{itemize} 
\item [(i)] The process $\hat Y$ that describes the (rcll version of the) decoration from the marked point $\uprho'$ back to the root $\uprho$ has the law $\hat Q_{x,1}$. 
\item [(ii)] The variables 
$\max Y$ and $x/\min Y$ have the same law.
\end{itemize}
\end{corollary}

\begin{proof} The first claim is plain from Theorem \ref{T1}, Corollary \ref{C:dual}, and the Lamperti transformation. The second follows easily from the identity 
$\max Y= \max \hat Y$ and again using the Lamperti transformation. 
\end{proof}
Let us further point out in this direction,  that the distribution of $\max Y$ under $Q_{1,x}$ can be determined, at least theoretically, combining the Lamperti transformation and fluctuation theory for L\'evy processes, and more specifically the Wiener-Hopf factorization.

We now conclude this section with a comment intended for readers having already a solid acquaintance  with the notions developed in \cite[Chapters 2 and 5]{ssMt}.
\begin{remark}The argument of the proof of Theorem \ref{T1} enables us  to describe more generally --although less explicitly--  the reproduction process along the marked segment $\llbracket \uprho, \uprho'\rrbracket$.
Note that here the right-extremity is included, in order to take into account the fringe  subtree rooted at $\uprho'$, that is $T'\subset T$ and $t'\in T'$ if and only if $\uprho'\in \llbracket \uprho,t'\rrbracket $. 
Recall from  \cite[Section 5.2]{ssMt} that under $\Q$,  if we mark a point $\uprho^{\bullet}$ in $T$ at random proportionally to $\uplambda^{\gamma_0}$, then the law of the
decoration-reproduction process along $\llbracket \uprho, \uprho^{\bullet}\rrbracket$ is characterized by the quadruplet $(\sigma^2, \mathrm{a}_{\gamma_0}, \boldsymbol{\Lambda}_{\gamma_0}; \alpha)$ given in \cite[Eq.(5.13) and Lemma 5.5]{ssMt}. The same calculation as in the proof of Theorem \ref{T1} shows that the  law of the
decoration-reproduction process along $\llbracket \uprho, \uprho'\rrbracket$ under $\Q^x$ can be expressed as the Doob transform of that along $\llbracket \uprho, \uprho^{\bullet}\rrbracket$, 
using the supermartingale $\mathrm{w}^{(\gamma_0)}(\log(x/ g(t)))$. One can further check that under $\Q^x$, the decorated fringe subtree rooted at $\uprho'$ is independent of 
 the decorated subtree pruned at $\uprho'$, and has the law $\Q_{x}$. 
\end{remark}

\section{Local time approximation of the harmonic measure}
\label{sec:convergenceH}

In short, our next purpose is to show that, under appropriate assumptions, the  harmonic measure $\upmu(\dd t)$ on the self-similar Markov tree arises as the normalized limit of the local time measures $L(x, \dd t)$  as $x\to 0+$. Recall from \cite[Section 2.3.2]{ssMt} that the existence of $\upmu$ requires  a  Cram\'er-type assumption for the cumulant function $\kappa$; see \cite[Assumption 2.13]{ssMt}. It will be convenient here to use the following slightly stronger 
version\footnote{Beware that  for the sake of simplification,  we use the notation $\omega$  here in place of $\omega_-$ in \cite{ssMt}.}, which we enforce throughout this section.

\begin{assumption}\label{A:omega-}
Suppose that there exist $\omega>0$ and $p\in(1,2]$ such that\begin{itemize}
\item[(i)]  $\kappa$ takes finite values on $[\omega, p\omega]$ with 
$ \kappa(\omega)=0$ and $\kappa'(\omega)\in(-\infty,0)$,
\item[(ii)] $ \int_{\mathcal{S}_1}\boldsymbol{\Lambda}_1(\mathrm{d} \mathbf{y} ) ~\left(\sum_{i=1}^{\infty} \e^{ y_i\omega}\right)^p< \infty$.\end{itemize}
\end{assumption}

Roughly speaking, the fact that $\omega$ is a root of the cumulant $\kappa$ yields a so-called intrinsic martingale \cite[Lemma 2.14]{ssMt}.
The further conditions in Assumption \ref{A:omega-} ensure the uniform integrability of this martingale, which in turn is the key for the construction of the harmonic measure $\upmu$. See \cite[Proposition 1.11 and Section 2.3.2]{ssMt} for details. 
We also recall from \cite[Section 2.3.3]{ssMt} that $\upmu$ arises as limit of weighted length measures \eqref{E:weightedlength}. Specifically, there exists a sequence of positive real numbers $(\gamma_n)_{n\geq 1}$ that decreases to $\omega$ such that, $\Q$-a.s., the sequence of finite measures 
$-\kappa(\gamma_n)  \uplambda^{\gamma_n}$ on $T$ converges weakly towards $\upmu$ as $n\to \infty$.  We also recall from  \cite[Eq.(2.30)]{ssMt} that 
\begin{equation}\label{equation:moyenne:harm:mu}
    \mathbb{E}_y(\upmu(T)) = y^{\omega}, \quad  \text{for }  y > 0.
\end{equation}

The purpose of  the section is to point at further approximations of $\upmu$ using local times and the following counting measures on $T$. To start with, introduce
for every $x>0$ the  first hitting line  of the decoration at level  $x$,
$$\mathcal{H}_x=\{t\in T: g(t)=x \text{ and } g(s)\neq x \text{ for all }s\prec t\},$$
where the notation $s\prec t$ means that $s$ belongs to the right-open segment $\llbracket \uprho, t \llbracket$ from the root to $t$. 
Then $\mathcal{H}_x$ is a finite set a.s. and we write $N(x, \dd t)$ for its counting measure on $T$; in particular $N(x,T)=\#\mathcal{H}_x$ . 
The normalizations appearing  in the claim below can be determined from Proposition \ref{P:exlt}(ii) and the forthcoming Lemma \ref{L:anotherstep}(iii).

\begin{theorem}\label{theorem:convergence}
Let Assumption \ref{A:omega-} hold. Then, there exists a sequence $(x_n)_{n\geq 1}$ of positive real numbers with $x_n\to 0$ as $n\to\infty$, such that, $\Q$-a.s. we have
    \begin{equation*}
        \lim_{n \to \infty} \frac{L(x_n, \dd t)}{\E(L(x_n, T))} =
        \lim_{n \to \infty} \frac{N(x_n, \dd t)}{\E(N(x_n, T))} = \upmu(\dd t)
    \end{equation*}
in the sense of weak convergence for finite measures on $T$.
\end{theorem}

The rest of this section is devoted to the proof of
Theorem~\ref{theorem:convergence}, which proceeds in several steps.
Let us first  stress that the convergence of the rescaled local time measures is
consistent with the results of the previous sections. Indeed, combining
Theorem~\ref{T1} with Proposition \ref{convergence:P:x:y}, it follows that the law of the spine decoration
$$
g_{\mid\llbracket\uprho,\uprho'\rrbracket},
\qquad\text{under }\Q^x,
$$
converges as $x\downarrow 0$ to a self-similar Markov process with
characteristics $(\kappa^{(\omega)},\alpha)$.
This limiting law is exactly the distribution of the decoration along the
spine from the root to a point sampled according to the harmonic measure
$\upmu$, after biasing by the total harmonic mass $\upmu(T)$
(see \cite[Chapter~5]{ssMt}).

However, this spine convergence by itself does not imply
Theorem~\ref{theorem:convergence} and we will not use it directly: to obtain weak convergence of the random
measures it is crucial to control their total masses. Actually, the key step of the proof consists in  proving the convergence of total masses, see Proposition \ref{P:cvnombrepar} below; 
the convergence of the measures themselves then follows by standard techniques. 
 We start with the following version of Proposition \ref{P:exlt} that will be needed in the sequel.
  \begin{proposition}\label{P:exltomega} We have:\begin{enumerate}
\item[(i)]
The shifted cumulant, 
$$\kappa^{(\omega)}(\cdot)=\kappa(\omega+\cdot)$$ is a L\'evy exponent with no killing for which \eqref{E:occden} and \eqref{E:reg} are satisfied.
Its potential density  $\mathrm{w}^{(\omega)}$ is continuous and satisfies
$$\lim_{y\to \infty} \mathrm{w}^{(\omega)}(y)=0 \quad \text{and} \quad \lim_{y\to -\infty} \mathrm{w}^{(\omega)}(y)=-1/\kappa'(\omega) .$$
\item[(ii)] For every $x,x'>0$, we have
$$\E_{x'}\big( L(x,T)\big) = (x'/x)^{\omega}\mathrm{w}^{(\omega)}(\log(x/x')).$$
\end{enumerate}
\end{proposition} 
  \begin{proof} The claims follow readily from Proposition \ref{P:exlt} and Lemmas  \ref{L:expofam:gamma:-} and \ref{lemma:asymv}.
  \end{proof}

 The notion of local decomposition for self-similar Markov trees, which should be thought of as a spatial version of the Markov branching property for self-similar Markov trees, plays a crucial role in this section. Let us give a  quick introduction; to keep the presentation short, we focus on a particular form and  refer to \cite[Chapter 4]{ssMt} for the general  theory and technical details. 
 
 First, recall that $ \mathbb{T}$ denotes the space of  isomorphism classes of decorated trees $(T,g)$  and is equipped with the probability measure $\Q$, which is the law of our ssMt.  
Given a closed subtree $T'\subset T$ with $\uprho\in T'$, we are interested in the family of connected components  of $T\backslash T'$.
Any such component, say $\uptau^*$, is attached to $T'$ at a unique point, say  $r\in T'$, and $r$ is the sole  boundary point of $\uptau^*$. 
So $\uptau=\uptau^*\cup \{r\}$ is another closed subtree of $T$, which we decorate with $g_{\uptau}: \uptau\to \R$ defined by
$$g_{\uptau}(y) = \left\{ \begin{matrix} g(y)  &\text{ \,  if }y\in \uptau^*,\\ \ell &\text{ if }y=r,\\
\end{matrix} \right.
$$
where 
$$\ell= \limsup_{\uptau^*\ni x\to r} g(x)$$
is called the germ\footnote{In general, $\ell<g(r)$.} of the decoration on $\uptau$.
Then $(\uptau ,g_{\uptau})$ is again a decorated tree, which is naturally rooted at $r$ and has initial decoration $\ell$; it is referred to as a subtree dangling from $T'$.

We say that the subtree $T'$ induces a local decomposition  if we can index the connected components of $T\backslash T'$ by some countable set $I$, say $(\uptau^*_i)_{i\in I}$, in such a way that under the law $\Q$ and conditionally on the restriction of the decorated tree to $T'$, $(T',g_{\mid T'})$ and on the family  $(r_i,\ell_i)_{i\in I}$, the decorated subtrees $(\uptau_i, g_{\uptau_i})$ for $i\in I$ are independent, and each $(\uptau_i, g_{\uptau_i})$ has the law $\Q_{\ell_i}$.

In what follows, we will apply this notion  for ${T}^\prime$ of the following form    for some  $0 < x < 1$:
    \begin{equation*}
        {T}_{> x} = \big\{ t \in T  : g(s)  >x\text{ for all } s\prec t  \big\},
    \end{equation*}
    and 
    \begin{equation*}
        {T}_{\neq x} = \big\{ t \in T : g(s)\neq x \text{ for all }s \prec t   \big\}.
    \end{equation*}
  It is readily checked from the same argument as in the proofs  of Propositions  4.3 and 4.9 in \cite[Chapter 4]{ssMt} that each induces a local decomposition, and we will use the notation $I_{>x}$ and $I_{\neq x}$ for the set of indices appearing in the respective local decomposition. We stress that 
  there is the identity $\mathcal {L}(x)=\{r_i: i\in I_{\neq x}\}$ and that 
  the germ of the decoration at any subtree component of $T\backslash T_{\neq x}$ is always $\ell=x=g(r)$.

   We now state a technical first-moment formula related to the local decomposition induced by $T_{>x}$, that  will be needed in the sequel.
   Recall that for every $i\in I_{>x}$, $r_i$ is the root of the component $\uptau_i$ and $\ell_i$ the associated germ.

\begin{lemma}\label{equation:manytoone} Let $F$ be  a non-negative functional  and $0 < x < 1$. We have 
 $$
            \E\Big( \sum_{i \in {I}_{> x}}   F(\ell_i ) \cdot \ell_i^{\omega}  \Big)
            = 
            E^{(\omega)} \left(F(X_{\varsigma_x}) \right),
$$
        where in the right-hand side,  
          $E^{(\omega)}$ stands for  the expectation for the law $Q^{(\omega)}$ of a self similar Markov process $X$ started from $1$ with characteristics  $(
        \kappa^{(\omega)}; \alpha)$, and $\varsigma_x=\inf\{t\geq 0: X_t<x\}$ is the first passage time below $x$.
        
       \end{lemma}

\begin{proof} The proof relies on the many-to-one formula  similar to that in the proof of Theorem \ref{T1}, but now for a point marked according to the harmonic measure $\upmu$. With the same notations as in \eqref{equation:spinalL}, we introduce a pointed measure $\Q^{(\omega)}$ characterized by the relation: 
\begin{equation} \label{equation:spinalmu}
    \E^{(\omega)}(G(T,g,\uprho')) = \E\left( \int_T G(T,g,t) \upmu( \dd t)\right),
\end{equation}
for every $F:\mathbb{T}^\bullet\to \mathbb{R}_+$ measurable function.  By \eqref{equation:moyenne:harm:mu}, $\Q^{(\omega)}$ is  a probability measure, and by \cite[Corollary 5.8]{ssMt}, 
 \begin{equation}\label{equation:spinalmub}
  \text{the distribution of  }   g_{\mid \llbracket \uprho, \uprho'\rrbracket }  \text{ under } \Q^{(\omega)} \text{ is } Q^{(\omega)}. 
 \end{equation}
We stress that  since $\kappa^{(\omega)}(0) = 0$, the self similar Markov process under $Q^{(\omega)}$ dies continuously at $0$.
 Thus for $\upmu$-almost every $t\in T$, we have $g(t)=0$ and  the decoration along the branch from the root to $t$ reaches levels below $x$.
  
 Getting back to \eqref{equation:spinalmu},  for $t\in T$ with $g(t)=0$, consider the unique  $s\in \llbracket \rho ,t\rrbracket $ such that $g(s)\leq x$ and $g(s^\prime)> x$ for every $s^\prime\in \llbracket \rho ,s\rrbracket\setminus \{s\}$. In words $s$ is the first ancestor with decoration smaller or equal to $x$.  We set $\ell_x(t)=g(s)$, and for definiteness by convention we take $\ell_x(t)=0$ if $g(t)\neq 0$.
   An application of  \eqref{equation:moyenne:harm:mu} combined with  the local decomposition induced by $\text{T}_{> x}$ yields: 
    \begin{equation} \label{equation:sum1}
    \E\left( \int_T  F(\ell_{x}(t) ) \upmu( \dd t)\right) = 
        \E\left(\sum_{i  \in {I}_{> x}} F(\ell_i )\upmu(\uptau_i) \right) = \E\left( \sum_{i  \in {I}_{> x} } F(\ell_i  ) \ell_i^{\omega} \right).
        \end{equation}
  On the other hand, we get  from  \eqref{equation:spinalmu} and \eqref{equation:spinalmub} that the left-hand side of  \eqref{equation:sum1} can be re-expressed as $  E^{(\omega)} \left(F(X_{\varsigma_x}) \right)$, which completes the proof.
\end{proof}

We then make another step towards Theorem \ref{theorem:convergence}.

\begin{lemma} \label{L:anotherstep}
 The following assertions hold.
 
 \begin{enumerate}
\item[(i)] 
    $  \lim_{x \to 0+} \sum_{i  \in {I}_{> x}} \ell_i^{\omega} = \upmu(T)$ in $L^1(\mathbb{Q})$. 
  \item[(ii)] The family $\big\{ x^{\omega}N(x,T) : x \in \R \big\}$  is uniformly integrable under $\Q$.
  
  \item[(iii)] Moreover, there is the identity
$$\E\big( N(x,T)\big) = x^{-\omega}\mathrm{w}^{(\omega)}(\log x))/\mathrm{w}^{(\omega)}(0),$$
and as a consequence, 
$$\lim_{x\to 0+} x^{\omega}   \E(N(x,T))  = - \frac{1}{\kappa'(\omega) \mathrm{w}^{(\omega)}(0)}.$$
\end{enumerate}
\end{lemma}
  \begin{proof} (i)
     Let $\mathcal{F}_{ > x}$ be the sigma-field generated by $(T_{ > x},g_{\mid T_{>x}})$ and $ (r_i, \ell_i)_{i \in I_{ > x}}$. 
     Note that $\upmu({T}_{> x}) = 0$ since the harmonic measure does not charge the skeleton of $T$ (see \cite[Proposition 1.12]{ssMt}).
      This fact paired with  the local decomposition induced by ${T}_{ > x}$ and \eqref{equation:moyenne:harm:mu} yields the identity
    \begin{equation*}
        \mathbb{E}\big( \upmu(T) \big| \mathcal{F}_{{ > x}} \big) = \sum_{i  \in {I}_{ > x}} \ell_i^{\omega}, \quad \mathbb{Q} \text{-a.s.}
    \end{equation*}
    Obviously $ \left(\mathcal{F}_{{ > x}} \right)_{x>0}$ is a descending filtration and  the process in the last display is a uniformly integrable martingale with terminal value $\upmu(T)$ (observe that $(T,g)$ is clearly $\mathcal{F}_{{ > 0+}}$-measurable). 
  
  (ii)   Let $\mathcal{F}_{ \neq x}$ be the sigma-field generated by $(T_{ \neq x},g_{\mid T_{\neq x}})$ and $ (r_i, \ell_i)_{i \in I_{ \neq x}}$.  Then, we have
      $$\sum_{i \in {I}_{\neq x}} \upmu(\uptau_i) \leq \upmu(T).$$
   Recall that $\ell_i = x$ for $i \in I_{\neq x}$ and $N(x,T) = \#I_{\neq x}$, so the local decomposition induced by $T_{\neq x}$ paired with \eqref{equation:moyenne:harm:mu} yields  
    \begin{equation*}
        x^{\omega}N(x,T) = \sum_{ i \in I_{\neq x}} \ell_i^{\omega} \leq \mathbb{E}\big( \upmu(T) \big| \mathcal{F}_{\neq x} \big).
    \end{equation*}
   This entails the claim of uniform integrability.  

(iii) We first write
$$L(x,T) = \sum_{i\in I_{\neq x}} L(x, \tau_i).$$
We next  take expectations using the local decomposition induced by $T_{\neq x}$.  
Recall from self-similarity that the law of $L(x,T)$ under $\Q_x$ does not depend on $x>0$, and   that $N(x,T) = \#I_x$. We get
\begin{equation}\label{equation:meanN(x,T)}
\E(L(x,T)) = \E(N(x,T)) \E(L(1,T)),
\end{equation}
and the claims follow from Proposition \ref{P:exltomega}. 
\end{proof}

We now have all the ingredients needed to establish the convergence of the normalized total masses. 
\begin{proposition} \label{P:cvnombrepar} For $y>0$, the following limits hold in $L^1(\Q_y)$:
$$    \lim_{x\to  0} \frac{L(x,T)}{\E(L(x,T))} =  \lim_{x\to  0} \frac{N(x,T)}{\E(N(x,T))} = \upmu(T). $$

\end{proposition}
\begin{proof} 
Plainly, it suffices to prove the result for $y = 1$, since by self-similarity, the distributions of $L(xy,T)$ and of $N(xy,T)$ under $\mathbb{Q}_y$ do not depend on $y>0$. We also point out that the first limit readily follows from the second.  Indeed, since  the local time measure $L(x, \dd t)$ is carried by the level set $\{t\in T: g(t)=x\}$, we can write 
$$L(x,T)=\sum_{i\in I_{\neq x}} L(x,\tau_i),$$
and we deduce from the local decomposition induced by $T_{\neq x}$ that the right-hand side is actually 
the sum of $N(x,T)$ i.i.d. variables distributed  as $L(1,T)$ under $\Q_1$.
By the law of large numbers, we have as $x\to 0+$ that there is the convergence in probability  under $\Q$, 
$$\lim_{x\to 0+}\frac{L(x,T)}{N(x,T)} = \E(L(1,T))=  \mathrm{w}^{(\omega)}(0).$$
We can easily conclude by an application of Scheff\'e's Lemma using Proposition \ref{P:exltomega}(ii) and Lemma \ref{L:anotherstep}(iii) that the first limit in the statement indeed holds in $L^1(\Q)$. 
  
 Local decompositions also play a major role in the proof of the second convergence. Beware however that we shall apply them to $T_{>\sqrt x}$ rather than $T_{>x}$.  We need the following slight modification of Lemma \ref{L:anotherstep}(i):
 \begin{equation} \label{E:modanother}
  \lim_{x \to 0+} \sum_{i  \in {I}_{> \sqrt x}} \ell_i^{\omega} \indset{\ell_i\geq x^{2/3}}= \upmu(T) \qquad \text{ in }L^1(\mathbb{Q}).
 \end{equation}
  In this direction, we start to observe from  the many-to-one formula of Lemma \ref{equation:manytoone} that 
  $$ \E\left( \sum_{i\in I_{>\sqrt x}} \ell_i^\omega \indset{\ell_i <x^{2/3}} \right)=Q^{(\omega)}\left( X_{\varsigma(\sqrt x)}< x^{2/3}\right).$$
 Recall from  Assumption \ref{A:omega-} that $\kappa'(\omega)\in(-\infty,0)$, so the L\'evy process with exponent $\kappa^{(\omega)}$ has a finite strictly negative expectation. The combination of the renewal theorem and the Lamperti transformation entails that  as $x\to 0+$,
 $x^{-1/2} X_{\varsigma(\sqrt x)}$ converges in distribution under $Q^{(\omega)}$, where the limiting variable is strictly positive a.s. 
 Therefore
 $$\lim_{x\to 0+} \E\left( \sum_{i\in I_{>\sqrt x}} \ell_i^\omega \indset{\ell_i <x^{2/3}} \right)= \lim_{x\to 0+} Q^{(\omega)}\left( x^{-1/2}X_{\varsigma(\sqrt x)}< x^{1/6}\right)=0,$$
  and Lemma \ref{L:anotherstep}(i) implies \eqref{E:modanother}.

Next, we write $N_i(x)=N(x, \tau_i)$ for every $i\in I_{>\sqrt x}$, so that 
        \begin{equation*}
      x^{\omega}N(x,T) 
      = \sum_{i \in I_{>\sqrt x}} x^{\omega}N_i(x) 
      =\sum_{i  \in {I}_{>\sqrt x}} \ell_i^{\omega}b_{i}(x), \qquad \text{where }b_i(x)= {\Big( \frac{x}{\ell_i} \Big)}^{\omega}N_i(x).
    \end{equation*}
    The local decomposition induced by $T_{>\sqrt x}$ shows that conditionally on $\mathcal{F}_{>\sqrt x}$, the variables $b_i(x)$ for $i\in I_{>\sqrt x}$ are independent,
    and the conditional law of each $b_i(x)$ is that of $y^{\omega}N(y)$ under $\Q$ for $y=x/\ell_i$.

    Observing from Lemma \ref{L:anotherstep}(iii) that $\E(y^{\omega}N(y))\leq 1$ for all $y>0$, we deduce from the calculation above that
    $$\lim_{x\to 0+} \E\left( \sum_{i  \in {I}_{>\sqrt x}} \ell_i^{\omega}b_{i}(x)\indset{\ell_i<x^{2/3}}\right)=0.$$
    We can now focus on indices  ${i  \in {I}_{>\sqrt x}}$ with $\ell_i\geq x^{2/3}$, so $x/\ell_i\leq x^{1/3}$, and we have from \eqref{E:modanother}  that
    $$ \lim_{x \to 0+} \sum_{i  \in {I}_{> \sqrt x}} \ell_i^{\omega} \frac{\mathrm{w}^{(\omega)}(\log(x/\ell_i))}{\mathrm{w}^{(\omega)}(0)} \indset{\ell_i\geq x^{2/3}}= -\frac{\upmu(T)}{\kappa'(\omega)\mathrm{w}^{(\omega)}(0)}, \qquad \text{ in }L^1(\mathbb{Q}).$$
    We know from Lemma \ref{L:anotherstep} that
   $\big\{ y^{\omega}N(y,T) : y \in \R \big\}$  is uniformly integrable under $\Q$, which  enables us to apply \cite[Lemma 9.2]{ssMt}, and we get 
  $$\lim_{x\to 0+} \sum_{i  \in {I}_{>\sqrt x}} \ell_i^{\omega}b_{i}(x)\indset{\ell_i\geq x^{2/3}} 
  = - \frac{\upmu(T)}{\kappa'(\omega) \mathrm{w}^{(\omega)}(0)}
   , \qquad \text{in }L^1(\Q).$$
  Putting the pieces together, we conclude that
  $$\lim_{x\to 0+} x^{\omega}N(x,T) = - \frac{\upmu(T)}{\kappa'(\omega) \mathrm{w}^{(\omega)}(0)}
   , \qquad \text{in }L^1(\Q),$$
which established the second limit of the statement. 
\end{proof}

Now that we have established the convergence of the normalized total masses of the local time measure, the proof of Theorem~\ref{theorem:convergence} follows by arguments similar to those in \cite[Section 2.3.3]{ssMt}. To simplify some arguments, throughout the remainder of this section we work under $\mathbb{P}$, noting that since all the quantities in Theorem~\ref{theorem:convergence} are intrinsic, it suffices to prove the theorem under $\mathbb{P}$.
Under $\mathbb{P}$ we can consider the segments $S_u$, $u \in \mathbb{U}$, introduced in Section~\ref{sec:nutshell} and  for every $u \in \mathbb{U}$, we write
$$ T_u=\bigcup_{v\in\mathbb{U},u\preceq v} S_v\subset T,$$
where, for elements of $u,v\mathbb{U}$, we write $u\preceq v$ if $u$ is an ancestor of $v$. The last missing step required to prove Theorem~\ref{theorem:convergence} consists in establishing the following convergence criterion.

\begin{lemma}\label{lemma:weakConv}
Under $\mathbb{P}$, let $(\nu_n)_{n \geq 0}$ be a sequence of measures on $T$ such that for every $u \in \mathbb{U}$, we have   
$\nu_n(T_u) \to \upmu(T_u)$ in probability. Then, there exists a sub-sequence $(n_k)_{k \geq 0}$ along which, $\mathbb{P}$-a.s., we have
$$\nu_{n_k}(\dd t) \to \upmu(\dd t),\quad k\to \infty,$$
in the sense of weak convergence for finite measures on $T$. 
\end{lemma}

\begin{proof}
In this proof,  we argue under $\mathbb{P}$. By a diagonal argument, we can extract a sub-sequence $(n_j)_{j \geq 0}$ along which, a.s., we have 
\begin{equation}\label{equation:convmasses}
        \lim_{j \to \infty} \nu_{n_j}(T_u) = \upmu(T_u), \quad \text{ for every } u \in \mathbb{U}.  
\end{equation}
To simplify notation we will assume that $(n_j)_{j \geq 0}$ is just the original sequence $(n)_{n \geq 0}$.
For every $n\geq 0$, we also introduce the subset
$$T^{(n)}:=\bigcup_{u\in \mathbb{U},\, |u|\leq n}  S_u,$$
and we recall from \cite[Proposition 1.12 and Lemma 2.14 ]{ssMt} that, a.s., $\upmu$ is a finite measure and
\begin{equation}\label{support:mu}
\upmu\big(\bigcup_{n\geq 0} T^{(n)}\big)=0.
\end{equation}
Let us consider an event $\mathcal{A}$ of full probability where $T$ is compact,  $\upmu(T)<\infty$ and  \eqref{equation:convmasses} and  \eqref{support:mu} hold. For the rest of the proof we work with a fixed realisation in $ \mathcal{A}$, and remark that to conclude it suffices to establish that on this realisation the sequence $(\nu_n)_{n\geq 0}$  converges weakly towards $\upmu$. First, since $\upmu(T)$ is finite,  the sequence of measures $(\nu_n)_{n \geq 0}$ is weakly pre-compact. Let $\nu$ be a possible weak limit, that is, there exists a sub-sequence along which $(\nu_n)_{n \geq 0}$ converges towards $\nu$. The rest of the proof is devoted to showing that $\nu = \upmu$, and again, to simplify notation, we suppose that the convergence towards $\nu$ holds along the original sequence.
    
\par Let us start with some preliminary remarks. First a straightforward application of  Portmanteau's theorem, combined with \eqref{equation:convmasses} and \eqref{support:mu},  entails that   as $\upmu$, the measure $\nu$ is supported on
 $$\partial T:=T\setminus \bigcup_{n\geq 0} T^{(n)}. $$   
Furthermore, $\upmu$ and $\nu$ also coincide on  $T_u$ for  $u \in \mathbb{U}$ since 
    \begin{equation}\label{equation_Tus}
        \lim_{n \to \infty} \nu_n(T_u)  = \nu(T_u), \quad \text{ for } u \in \mathbb{U},   
    \end{equation}
    as a consequence of Portmanteau's theorem paired with the fact that $\nu$ is supported on $\partial T$. We will now infer from these properties that  $\upmu$ and $\nu$ coincide on every closed set of $T$, which will complete the proof. To this end, fix an arbitrary closed set  $C$ and write $\partial C = C \cap \partial T$. We already know that $\upmu$ and $ \nu$ do not charge $C\setminus \partial C$, so it remains to check that $\upmu(\partial C) = \nu(\partial C)$. In this direction, for every $n \geq 0$ we let $\mathbb{U}^n_C$ be the subset  $\{u \in \mathbb{U}^n : T_u \cap C \neq \emptyset\}$. By \eqref{sum:finite} it is straightforward to verify that  
    \begin{equation}\label{equation:eq}
        \partial C = \bigcap_{n \geq 0} \bigcup_{ u \in \mathbb{U}^n_C } T_u. 
    \end{equation}
Since the sequence of sets $\bigcup_{ u \in \mathbb{U}^n_C } T_u,$  $n \geq 0,$ 
    is non-increasing and the sets $T_u\cap \partial T$, $u \in \mathbb{U}^n$, are disjoint, we infer from \eqref{equation:eq} that 
    \begin{align*}
        \nu(C) = \lim_{n \to \infty } \nu\Big( \bigcup_{ u \in \mathbb{U}^n_C } T_u \Big) = \lim_{n \to \infty} \sum_{u \in \mathbb{U}_C^n} \nu(T_u), 
    \end{align*}
    and the same equalities hold if we replace $\nu$ by $\upmu$. 
    Recalling from \eqref{equation_Tus} that  $\nu(T_u) = \upmu(T_u)$ for $u \in \mathbb{U}^n_C$, we deduce that $\nu(C) = \upmu(C)$. This proves that $\upmu$ and $\nu$ coincide on every closed set and therefore are identical.  
\end{proof}

We now have all the necessary ingredients needed to conclude the proof of the main result of the section. 

\begin{proof}[Proof of Theorem \ref{theorem:convergence}] We shall only prove the first convergence in Theorem \ref{theorem:convergence} since the second convergence follows applying analogous arguments and we recall that for simplicity we work under $\P$. In this direction, for $x>0$, we write $m(x)= x^{-\omega} \mathrm{w}^{(\omega)}(\log x)$ and remark that, by Proposition \ref{P:exltomega}   and Lemma \ref{lemma:weakConv}, it suffices to establish:
\begin{equation}\label{eq:L:T:u:upmu}
\lim_{x\downarrow  0}        \mathbb{E}\Big( \big|\frac{L(x,T_u)}{m(x)} - \upmu(T_u) \big| \wedge 1 \Big)=0, 
\end{equation}
for every $u\in \mathbb{U}$. For the remaining of the proof we fix $u\in \mathbb{U}$. Observe that, by the discussion at the end of Section \ref{sec:construc} and the fact that the harmonic measures is a limit of rescaled weight measures,   the distribution of $ (L(x,T) , \upmu(T))$ under $\mathbb{P}_y$ coincides with the one of  $(L(x/y,T) , y^{\omega} \upmu(T))$ under $\mathbb{P}$. By conditioning on the first $|u|$ generations,  the branching property yields: 
$$ \mathbb{E}\Big( \big|\frac{L(x,T_u)}{m(x)} - \upmu(T_u) \big| \wedge 1 \Big)= \E\Big(\mathbf{1}_{\chi(u)>0}G\big(x,\chi(u)\big)\Big),$$
where 
$$G(x,y):= \mathbb{E}\Big( \big|\frac{m(x/y)}{m(x)}\cdot \frac{L(x/y,T)}{m(x/y)} - y^{\omega}\upmu(T) \big| \wedge 1 \Big),\quad x,y>0.$$
The function $G$ is bounded by $1$, and by Propositios \ref{P:exltomega}  and  \ref{P:cvnombrepar} we have:
     \begin{equation*}
         \lim_{x \to 0}\frac{m(x/y)}{m(x)}= y^{\omega}, \quad \text{ as well as  } \quad \lim_{x \to 0}  \frac{L(x/y,T)}{m(x/y)} = \upmu(T), \text{ in }  L_1(\mathbb{P}), 
     \end{equation*}
for every $y>0$.
     Hence, by an application of dominated convergence theorem we derive \eqref{eq:L:T:u:upmu} as desired.
\end{proof}

\section{Excursions away from a level set}    
 \label{sec:branchstruc}
 
Our main purpose in this section is to investigate the level sets
$$ \mathcal L(x)=\{t\in T: g(t)=x\}, \qquad x>0,$$
as well as the excursions of the decoration away from $x$.
For the sake of simplicity, we work under $\P=\P_1$ when the decoration at the root is $g(\uprho)=1$; we further  focus on the level  $1$ and write 
$$\mathcal L = \mathcal L(1).$$
We stress that level sets $\mathcal L(x)$ for any $x>0$ can then be treated by an application of the local decomposition induced by the sub-tree $T_{\neq x}$ in the notation of Section \ref{sec:convergenceH}, and that changing the initial decoration to an arbitrary $y>0$ is of course  an easy matter by self-similarity. 

We call the closure of a connected component of  $T\backslash \mathcal L$ an excursion-subtree. Each excursion-subtree is naturally rooted at the closest point to $\uprho$, called the debut  of the excursion,  and further decorated by the restriction of $g$; we then simply call  a decorated excursion-subtree an excursion. Observe that by construction,
the decoration of an excursion at its root always equals $1$. We will soon observe that excursions have also an end and use this end as a natural mark.

We will start the study of excursions with the analysis of those that meet the ancestral segment. Their description involves a stopped Poisson point process, which notably enables us to define the so-called  excursion measure $\mathbf{N}\SB$. We will next show that  the level set $\mathcal L$  equipped 
with the local time distance has the structure of  a continuous branching tree; this should be viewed as an analog of \cite[Theorem 5.1]{RRO1} for Markov processes indexed by a L\'evy tree. Finally, we will describe  the family of excursions in terms of a Poisson point process, which is of course reminiscent of the classical result of K. It\^{o} for the excursions of a Markov process away from its starting point, and also of \cite[Theorem 3.8]{RRO2} for Markov processes indexed by a L\'evy tree.

\subsection{Excursions debuting on the ancestral segment}
Recall from Section \ref{sec:nutshell} that we can think of an ssMt as representing a population that evolves when time passes. 
The ancestor is labelled by $\varnothing$ and associated to an oriented closed segment $S\svn \subset T$ which we refer to as the ancestral segment and whose left extremity is the root $\uprho$.
We also denote  the right extremity of $S\svn$ by $\uprho\SB$, so $S\svn=\llbracket \uprho, \uprho\SB \rrbracket$, and use $\uprho\SB$ as a marked point on the self-similar Markov tree.  We recall the notation $\T\SB$ for the space of isomorphism classes
 of marked decorated trees that was discussed at the end of \cite[Section 1.4]{ssMt}. We shall also often view $S\svn$ as the closed interval $[0, z\svn]$, where $z\svn=|S\svn|$ is the lifetime of the ancestor. The origin $0$ is then identified with  $\uprho$ and the lifetime $z\svn$
with $\uprho\SB$.

The decoration $f\svn$ along $S\svn$ describes the size of the ancestor as a function of its age; its dynamics are that of a self-similar Markov process starting from $1$ and characterized by the pair $(\psi;\alpha)$,
where $\psi$ is the L\'evy exponent of the underlying  L\'evy process $\xi$ in the Lamperti transformation. 
We write
$$\mathcal L\svn=S\svn \cap \mathcal L = \{t\leq z\svn: f\svn(t)=1\},$$
and for simplicity, since the level $1$ has been fixed in this section,
$$L\svn (t)= L(1, \llbracket \uprho, t\rrbracket), \qquad t\in S\svn,$$
for the local time process at level $1$ of the decoration on the ancestral segment.
We also define the right-continuous inverse local time
$$L^{-1}\svn (a)= \inf\{t\geq 0: L\svn(t)>a\}, \qquad a\geq 0,$$
with the -unusual but convenient for our current purpose- convention that 
$$L^{-1}\svn (a)= z\svn, \qquad \text{for }a\geq L\svn(z\svn).$$
We recall a few basic facts in this setting.

\begin{lemma}\label{L:wellknown} The following assertions hold $\P$-a.s.
\begin{enumerate}
\item[(i)]  The decoration $f\svn$ is continuous on $\mathcal L\svn$.
\item[(ii)] $\mathcal L\svn$ is a perfect closed set with Lebesgue measure $0$.
\item[(iii)] The family of the open interval components of $[0, z\svn) \backslash \mathcal L\svn$
is given by
$$\big\{ \left( L^{-1}\svn (a-), L^{-1}\svn (a) \right): a>0 \text{ with } L^{-1}\svn (a)>L^{-1}\svn (a-)\big\}.$$
\end{enumerate}
\end{lemma}

 \begin{proof} Let us briefly recall the proof of (i), since a slightly modified argument will be also useful later on. 
 By the Lamperti transformation, the claim amounts to the well-known fact that a.s. for the L\'evy process $\xi$,
 if $t>0$ is such that either $\xi_{t}=0$ or $\xi_{t-}=0$, then actually $\xi_t=\xi_{t-}=0$. 
 
 Fix any $\varepsilon>0$, consider the times at which $\xi$ has a jump with absolute size $|\Delta \xi|>\varepsilon$, and write $\xi''$ for the compound Poisson process 
 obtained by summing those jumps as time passes.  It follows from the L\'evy-It\^{o} decomposition and standard properties of Poisson point processes that $\xi'=\xi-\xi''$ is also a L\'evy process, which is independent of  $\xi''$. Fix any $n\geq 1$ and consider the $n$-th jump time of $\xi''$, say $\varsigma$. Recalling from a well-known result due to Hartman and Wintner \cite[Theorem 27.16]{Sato} that the distribution of $\xi'_t$ is continuous for any $t>0$, we now see that the same holds for $\xi_{\varsigma}=\xi'_{\varsigma} +  \xi''_{\varsigma}$
 and $\xi_{\varsigma-}=\xi'_{\varsigma} +  \xi''_{\varsigma-}$; in particular the probability that $\xi_{\varsigma}=0$ or $\xi_{\varsigma-}=0$ is null.
    \end{proof}

The contribution of the ancestor to the population is encoded by its reproduction measure $\eta\svn(\dd s, \dd x)$, a point process on $S\svn \times \R_+$ 
that records the birth-time and the germ decoration of its children. We now observe that the ancestor does not reproduce when its decoration equals $1$.
\begin{lemma}\label{lemma:puntosx} One has 
 $\eta\svn(\mathcal L\svn\times \R_+) = 0$,  $\P$-a.s.
\end{lemma}
\begin{proof} 
 The construction  of the decoration-reproduction process $(f\svn, \eta\svn)$ of the ancestor in \cite[Section 2.2]{ssMt} 
 consists in applying the Lamperti time-substitution to a pair $(\xi, \mathbf{N})$, where $\xi$ is a L\'evy process and $\mathbf{N}$ a  Poisson point process on $\R_+\times ([-\infty, \infty)\times \mathcal S_1)$ with homogeneous intensity $\dd t \times \boldsymbol{\Lambda}(\dd y , \dd \mathbf{y})$. 
 Consider any Borel set $B\subset [-\infty, \infty)\times \mathcal S_1$ with $\boldsymbol{\Lambda}(B)<\infty$ and write $\varsigma$ for the $n$-th time at which $\mathbf{N}$
 has an atom in $B$. The same argument as in the proof of Lemma \ref{L:wellknown}(i) shows that the probability that either  $\xi_{\varsigma-}=0$
 or $\xi_{\varsigma}=0$ equals $0$. By sigma-finiteness of the generalized L\'evy measure $\boldsymbol{\Lambda}$, this yields our claim.  
\end{proof}

We point at  the following important consequence of the preceding lemmas for the analysis of excursions.
\begin{corollary} \label{C:debut} $\P$-a.s., if\: $T'$ is an excursion-subtree with $T'\cap S\svn\neq \emptyset$, then
its debut $d'$  belongs to $\mathcal{L}\svn$ and $T'\cap S\svn=\llbracket d', d''\rrbracket$ for some $d''\neq d'$. We then call $d''$ the end of the excursion-subtree.
\end{corollary}
\begin{proof} If $d'\not\in S\svn$, then for any point $t\in T'\cap S\svn$, the segment
$\llbracket d', t\rrbracket$ would be included in $T'$ and would thus contain a point closer to $\uprho$ than $d'$. This is a contradiction, so $d'\in S\svn$.
If $d'\not \in \mathcal{L}\svn$, 
 that is if $f\svn(d')\neq 1$, then from Lemma \ref{L:wellknown}(ii), there would be an open neighborhood of $d'$ in $S\svn$ on which the decoration 
 would avoid $1$ and again $d'$ would not be the closest point of $T'$ to the root. Thus $d'\in \mathcal{L}\svn$, and we know from Lemma \ref{lemma:puntosx} that $d'$ is not a branch-point in $T$.
So the intersection $T' \cap S\svn$ is actually a non-degenerate closed segment in $S\svn$; its right-end is the point $d''$ in the statement. 
 \end{proof}
 
We are interested in the excursions that meet the ancestral segment, and roughly speaking, Corollary \ref{C:debut} incites us to  log the tree $T$ at each point of $\mathcal L\svn$. Let $p\svn: T\to S\svn$ denote the projection on the ancestral segment. For any $a> 0$ with $L^{-1}\svn (a)>L^{-1}\svn (a-)$, 
write 
$$T\svn^{(a)}=\big\{t\in T: p\svn(t)\in \llbracket L^{-1}\svn (a-),L^{-1}\svn (a)\rrbracket\big\},$$
and $g\svn^{(a)}$ for the restriction of $g$ to $T\svn^{(a)}$. As usual we equip $T\svn^{(a)}$ with the restricted distance and we also endow it with the  root  $\uprho\svn^{(a)}= L^{-1}\svn (a-)$ and the mark  $\uprho\svn^{(a\bullet)}=L^{-1}\svn (a)$. We call $\llbracket \uprho\svn^{(a)}, \uprho\svn^{(a\bullet)}\rrbracket\subset S\svn$ the marked segment, and 
 write $\mathtt{T}\svn^{(a\bullet)}\in \T\SB$ for the isomorphism class of the marked decorated subtree $(T\svn^{(a)}, g\svn^{(a)}, \uprho\svn^{(a)}, \uprho\svn^{(a\bullet)})$. Roughly speaking, the roots and the marks on those subtrees enables us to reconstruct the original marked ssMt $\mathtt{T}\SB$ from the family 
 $\{\mathtt{T}\svn^{(a\bullet)}: 0< a \leq L\svn(z\svn)\}$, much in the same way as the original Markov process is recovered from its excursion process in It\^{o}'s synthesis.

 We stress that for $a<L\svn(z\svn)$,
the decoration $g\svn^{(a)}$ takes the value $1$ on the marked segment exactly at its extremities.
 On the other hand, for $a=L\svn(z\svn)$, we have $L^{-1}\svn (a-)=\sup\{t\leq z\svn: f\svn(t)=1\}$
 and $L^{-1}\svn (a)= z\svn$. So 
 $\mathtt{T}\svn^{(L\svn(z\svn)\bullet)}$ is  the fringe subtree rooted at the last passage of the decoration at level $1$ on the ancestral segment and further marked at $\uprho\SB$; we stress that the decoration at the marked point now differs from $1$.
  
 Recall that under $\P$, the total local time on the ancestral segment, $L\svn(z\svn)$, is exponentially distributed with mean
$v(0)$, where $v$ is  the potential density of the L\'evy process with exponent $\psi$. We now observe that the family of decorated subtrees that results from logging 
can be described in terms of a homogeneous Poisson random measure on $[0,\infty)\times \T\SB$.

\begin{lemma} \label{L:excancestral} Under $\P$, the point process\footnote{Implicitly, the sum only involves $a$'s with $L^{-1}\svn (a-)<L^{-1}\svn (a)$ and $a\leq L\svn(z\svn)$. The same convention applies as well in the sequel.}
$$\sum_{a>0} \delta_{(a,\mathtt{T}\svn^{(a\bullet)})}$$
is a homogeneous Poisson measure on $\R_+\times \T\SB$  stopped at  the first time when the decoration component evaluated at the marked point differs from $1$.
\end{lemma} 
\begin{proof} The heart of the proof relies on the following regenerative property at inverse local times. 
Fix $a>0$ and observe that $L^{-1}\svn (a)$ is a stopping time in the natural filtration of the ancestral
decoration process. We now work conditionally on $a<L\svn(z\svn)$.
We know from \cite[Proposition 4.8]{ssMt} that the so-called hull $\{t\in T: p\svn(t) \leq L^{-1}\svn (a)\}$
induces a local decomposition.  Furthermore, since by Lemma \ref{lemma:puntosx}, $L^{-1}\svn (a)$ is not a time at which the ancestor reproduces, $L^{-1}\svn (a)$ is the root of a unique dangling subtree, namely the  fringe subtree rooted at $L^{-1}\svn (a)$. Plainly, the germ of the decoration of this fringe subtree is $1$. 
We infer from the local decomposition that the decorated fringe subtree rooted at $L^{-1}\svn (a)$ is distributed as $(T, g, \rho,\rho^\bullet)$, under  $\mathbb{P}$, and is independent of the point measure $\sum_{b\leq a} \delta_{\mathtt{T}\svn^{(b\bullet)}}$.  The proof is readily completed; see e.g. \cite[Sections XII.1-2]{RY} and notably Theorem XII.2.4 there.
\end{proof} 

Our next purpose is to  define a natural measure $\mathbf N\SB$ on $\T$ carried by the subset of marked decorated trees such that the decoration is never $1$ on the interior of the tree. We will argue in the final section that $\mathbf N\SB$ serves as the excursion measure of the self-similar Markov tree away from the level $1$, much in the same way as for It\^{o}'s excursion theory for Markov processes.
For this, we introduce a reduction operator that consists in pruning marked trees at the first points where their decoration return to $1$.
More precisely, the reduction operator will be applied to marked decorated trees for which the decoration never takes the value $1$ on the open marked segment;
we stress that we do not want to cut such a tree at its root even when the decoration there is $1$, but rather at points different from the root at which  the decoration returns to $1$ for the first time. Specifically,  for any marked decorated tree $\mathtt T\SB=(T,g, \uprho, \uprho\SB)$ such that $g$ never takes the value $1$ on the open marked segment $\rrbracket  \uprho, \uprho\SB \llbracket$,
 we write $R(\mathtt T\SB)= (R(T), g_{\mid R(T)}, \uprho, \uprho\SB)$, where
$$R(T)=\{t\in T: g(s)\neq 1\text{ for every }s\in \rrbracket \uprho, t\llbracket\}.$$
We now introduce the measure $\mathbf N\SB$ defined for any measurable set $A\subset \T\SB$ by
\begin{equation}\label{E:defQ}
\mathbf N\SB(A)=\frac{1}{v(0)}
 \E\left( \#\{a>0: R(\mathtt T\svn^{(a\bullet)})\in A\} \right).
\end{equation}
In words, up to the normalizing factor $1/v(0)$ which stems from $\E(L\svn(z\svn))=v(0)$,
$\mathbf N\SB$ is the intensity measure of the point process induced by the family of reduced trees  $R(\mathtt T\svn^{(a\bullet)})$. 
Observing that the length of the marked segment is $L^{-1}\svn (a)-L^{-1}\svn (a-)>0$, we readily 
see that $\mathbf N\SB$ is sigma-finite.

\begin{corollary} \label{C:excancestral} Under $\P$, the point process of excursions met on the ancestral segment, 
$$\sum_{a>0} \delta_{(a,R(\mathtt{T}\svn^{(a\bullet)}))},$$
is a homogeneous Poisson measure on $\R_+\times \T\SB$  with intensity $ \dd t \times \mathbf N\SB$
stopped at  the first time when the decoration component evaluated at its marked point differs from $1$.
\end{corollary} 
\begin{proof} We use the elementary  fact that for any measurable space $E$ equipped with a sigma-finite measure $\mu$,
if $\sum \delta_{(t_i, x_i)}$ is a  homogeneous Poisson point process on $\R_+\times E$  with intensity $ \dd t \times \mu$,
and if $\varsigma$ is a stopping time  with finite mean in the natural filtration of this process, then  the intensity measure on $E$ of the stopped point process $\sum_{t_i\leq \varsigma} \delta_{x_i}$ is $\E(\varsigma) \mu$.
The mapping theorem for Poisson random measures combined with Lemma \ref{L:excancestral} and the identity $\E(L\svn(z\svn))=v(0)$
complete the proof.
\end{proof}

We distinguish four possible patterns of excursions $R(\mathtt{T}\svn^{(a\bullet)})$, depending on whether the decoration at the marked point equals $1$ (that is when 
$0<a<L\svn(z\svn)$) or differs from $1$ (that is when $a=L\svn(z\svn)$), 
and on whether the reduction operation is or not degenerate. This is especially relevant to analyze 
the number of points away from the marked segment at which the decoration of an excursion takes the value $1$, which plays an important role in the next sections.
In this direction, we now conclude the present one by pointing at another useful local decomposition (again, see \cite[Section 4.1]{ssMt} for background).
Consider the subtree of $T$ given by
$$F\svn = \{t\in T: g(s)\neq 1\text{ for all } p\svn(t) \prec s \prec t\}.$$
\begin{lemma}\label{L:locdecsvn} Under $\P$, the subtree $F\svn $ induces a local decomposition.
\end{lemma}
\begin{proof} Recall that for $i\geq 1$, $T_i$ denotes the subtree of $T$ generated by the $i$-th child of the ancestor, and let $\uprho_i\in S\svn$ be its root. Write $T_{i, \neq 1}$ the sub-subtree 
obtained by pruning  $T_i$ at the first points (if any) at which its decoration equals $1$. If $T_{i, \neq 1}= T_i$, that is if no pruning is performed, 
we set $k(i)=0$. Otherwise, we enumerate  (using some algorithm) the $k(i)\geq 1$ points at which $T_i$ is pruned; we  get $\uprho_{i1}, \ldots, \uprho_{ik(i)}$.
We also set $\ell_{ij}=1$ and write $\uptau_{ij}$ for the fringe decorated subtree rooted at $\uprho_{ij}$ for every $1\leq j\leq k(i)$. For definitiveness we let $\ell_{ij}=0$, $\uprho_{ij}=\uprho_i$ and agree that $\uptau_{ij}$ is degenerate for $j\geq k(i)+1$. 

We next observe that $F\svn$ is obtained by gluing the pruned subtrees $T_{i, \neq 1}$ at the locations $\uprho_i$ on the ancestral segment $S\svn$,
and naturally decorate $F\svn$ with the restriction of $g$. Combining the fact that $T_{\neq 1}$ induces a local decomposition under $\P_x$, for $x\neq 1$,  
and the local decomposition induced by the ancestral segment under $\P$, we can now check that the quantities defined above form indeed  a local decomposition indexed by $I=\N^2$ under $\P$. \end{proof}

\subsection{Branching structure of level sets of individuals} \label{S:Bsll}
Recall that from the point of view of a self-similar Markov tree as  population model, every individual $u\in\U$ is represented by a segment $S_u$ decorated by a function $f_u$ and further endowed with a reproduction process $\eta_u$.
It is natural to decompose the level set of the decoration (at level $1$)  in the whole tree as
\begin{equation} \label{E:decomp}
\mathcal L = \bigcup_{u\in \U} \mathcal L_u,
\end{equation}
where
$$\mathcal L_u=\{t\in S_u: f_u(t)=1\}$$
is the level set for the individual $u$,
and to introduce the subset of individuals
$$\mathbb L = \{u\in \U: \mathcal L_u \neq \emptyset\}.$$
Expressing the total local time $L(1,T)$ in the form
$$L(1,T) = \sum_{u\in \U} L(1,S_u)$$
and applying the Markov property of the decoration on each segment $S_u$, we
 easily deduce from Proposition \ref{P:exlt}  that
\begin{equation}\label{E:Lfinite}
\E(\# \mathbb L) =\mathrm{w}^{(\gamma_0)}(0)/v(0)<\infty.
\end{equation}
 
We  point at the following immediate consequences of Lemmas \ref{L:wellknown} and \ref{lemma:puntosx} combined with  Corollary \ref{C:debut}.
 
\begin{corollary} \label{C:Lnull} The following assertions hold $\P$-a.s.
\begin{enumerate}
\item[(i)]  For every $u\in\U$, the decoration $f_u$ on the segment $S_u$ is continuous on $\mathcal L_u$.
\item[(ii)] $\mathcal L$ is a perfect closed set with zero length measure,  $\uplambda (\mathcal L)=0$.
\item[(iii)] $\mathcal L$ avoids all the branching points of $T$.
\item[(iv)] If $d'$ is the debut of some excursion-subtree $T'$, then there is a unique $u\in \mathbb L$ such that $d'\in \mathcal{L}_u$,
and $T'\cap S_u=\llbracket d', d''\rrbracket$ for some $d''\neq d'$. Recall that  $d''$  is referred to as the end of the excursion-subtree $T'$. 
\end{enumerate}
\end{corollary}

\begin{proof} Fix any $u\in \U$, $u\neq \uprho$. We see from an application of the Markov property of decoration-reproduction processes (see \cite[Lemma 4.7]{ssMt})  at the first hitting time of $1$ for the decoration and Lemma \ref{L:wellknown} that $\mathcal L_u$ is a perfect closed set with length measure $0$ and that $f_u$ is continuous on $\mathcal L_u$. We deduce similarly from Lemma  \ref{lemma:puntosx} that the individual $u$ does not reproduce on $\mathcal L_u$.
As we already know from \eqref{E:Lfinite} that $\mathbb L$ is finite a.s., this proves (ii). 

Turning our attention to (iv), consider any excursion-subtree $T'$. If $T'\cap S\svn\neq \emptyset$, then we can apply  Corollary \ref{C:debut}.
Else, $T'$ is contained in one (and only one) of the children subtrees, $T_j$ for some $j\geq 1$. By the branching property, we can repeat the argument  generation after generation, until we find the first individual, say $u$, such that $T'\cap S_u\neq \emptyset$,
and we apply  Corollary \ref{C:debut}. We see in particular that $u\in \mathbb L$. The remaining assertions follows similarly by the branching property applying Lemmas \ref{L:wellknown} and \ref{lemma:puntosx}. 
\end{proof}

We next observe that $\mathbb L$ inherits a genealogical structure from $\U$, namely for any $u,v\in \mathbb L$, $v$ descends from $u$ if and only if $u\prec v$ in $\U$.
We readily deduce from the local decomposition of Lemma \ref{L:locdecsvn}
that $\mathbb L$ is a Galton--Watson tree. Its reproduction law  can be described in terms of the quantities analyzed in the preceding section, as we shall now explain.

For a generic marked decorated tree $\mathtt T\SB$, write $Z(\mathtt T\SB)$ for the number of points away from the marked segment at which the decoration equals $1$ for the first time.
Recall also the notation $\mathtt T\svn^{(a\bullet)}$ for the chunk of the marked ssMt that results by logging the ancestral segment at $L^{-1}\svn(a-)$ and $L^{-1}\svn(a)$. 
Then the quantity
$$\sum_{a>0} Z(\mathtt T\svn^{(a\bullet)})$$
is precisely the number of children of the ancestor $\varnothing$ in $\mathbb L$. We record the following immediate consequence.

\begin{corollary} \label{C:Zint}  We have
$$\int_{\T\SB} Z(\mathtt{T}\SB) \mathbf N\SB(\dd \mathtt{T}\SB) = \frac{1}{v(0)} -\frac{1}{\mathrm{w}^{(\gamma_0)}(0)}< \frac{1}{v(0)}.$$
\end{corollary} 
\begin{proof} By definition of $\mathbf N\SB$ and the reduction operator, we have the identity
$$v(0) \int_{\T\SB} Z(\mathtt{T}\SB) \mathbf N\SB(\dd \mathtt{T}\SB)  = \E\left( \sum _{a>0} Z(\mathtt T\svn^{(a\bullet)})
\right).$$
We have argued that the right-hand side is the mean reproduction of the Galton--Watson process with total population $\mathbb L$,
and the stated formula readily follows from \eqref{E:Lfinite}.
\end{proof}

We next turn our attention to the finer branching structure for level sets and not just individuals. 
It is natural to endow $\mathcal L$ with the pseudo-distance induced by the local time measure $L(1, \dd t)$,
$$d_L(s,t)= L(1, \llbracket s, t\rrbracket), \qquad s,t \in \mathcal L,$$
and write $(\tilde{\mathcal L}, \tilde d_L)$ for the metric space that is obtained by identifying any two points $s,t\in \mathcal L$ with $d_L(s,t)=0$. Specifically, for $s,t \in \mathcal L$, we write $s\sim t$ iff $\tilde{d}_L(a,b)=0$, which defines an equivalence relation on $T$.  The space $\tilde{\mathcal L}$ is the quotient space of $\mathcal L$ by this equivalence relation equipped with the metric (induced by) $\tilde{d}_L$ and rooted at the equivalent class of the root of $T$.

Let us discuss the identification a bit further.  Take some individual $u\in \mathbb L$,  and  consider an excursion from the point of view of $u$, say in the obvious notation,
$R(\mathtt T_u^{(a\bullet)})$ for some $a>0$. Then all the points in that excursion tree which belong to the level set $\mathcal L$  are identified one with the other.
More precisely, if $a<L_u(z_u)$ and if $Z(\mathtt T_u^{(a\bullet)})=0$, then the root $\uprho_u^{(a)}$ 
and the marked point $\uprho_u^{(a\bullet)}$ are the only two points in that excursion that belong to $\mathcal L_u$, and they are identified by $d_L$. 
If $a<L_u(z_u)$ again but now  $Z(\mathtt T_u^{(a\bullet)})=j\geq 1$, $j$ further points on the excursion are identified with the root and the marked point, which produces a branch point in $\tilde{\mathcal L}$.
These $j$ points do not belong to $\mathcal L_u$, but rather are given by the first points of level sets $\mathcal L_v$ on the segment $S_v$ for some children $v$ of $u$ in $\mathbb L$.
 Finally, the analysis for an ultimate excursion, that is when $a=L_u(z_u)$, or equivalently when the decoration at the marked point differs from $1$, 
is similar.

 It should now be plain that the identification induced by $d_L$ turns each individual level set $\mathcal L_u$ into a segment $\tilde{\mathcal L}_u$ with length $L_u(z_u)=L(1,S_u)$,
and $\tilde{\mathcal L}$ is a rooted continuous tree with a finite combinatorial structure. Its branching points with outer-degree $j\geq 2$ correspond either to excursions
$\mathtt T_u^{(a\bullet)}$ for $a<L_u(z_u)$ with $Z(\mathtt T_u^{(a\bullet)})=j-1$, or to ultimate excursions $\mathtt T_u^{(a\bullet)}$ for $a=L_u(z_u)$ with $Z(\mathtt T_u^{(a\bullet)})=j$. In turn its leaves 
correspond to ultimate excursions $\mathtt T_u^{(a\bullet)}$ for $a=L_u(z_u)$ with $Z(\mathtt T_u^{(a\bullet)})=0$.
We further stress that ultimate excursions $\mathtt T_u^{(a\bullet)}$ for $a=L_u(z_u)$ with $Z(\mathtt T_u^{(a\bullet)})=1$
are neither leaves nor branching points of $\tilde {\mathcal L}$, but rather ordinary points of outer-degree $1$, even though their root $\uprho_u^{(a)}$ appear as the last point of $\mathcal L_u$.  Also, remark that it follows from the construction that if we denote the length measure of $\tilde{\mathcal{L}}$ by $\uplambda_{\tilde{\mathcal{L}}}$ then we have:
\begin{equation}\label{eq:length:tilde}
\uplambda_{\tilde{\mathcal{L}}}\big(\tilde{\mathcal{L}}\big)= L(1, T), \quad \P\text{-a.s.} 
\end{equation}
\par   Let us stress that since the support of the measure $L(1,\mathrm{d}t )$ is precisely $\mathcal{L}$, we could equivalently obtain $(\tilde{\mathcal L}, \tilde d_L)$  by defining $d_L$ on the tree $T$, instead of solely on $\mathcal{L}$, and by identifying every pair of points $s,t\in T$ such that $L(1, \llbracket s, t\rrbracket)=0$. This second point of view yields that $(\tilde{\mathcal L}, \tilde d_L)$ 
 it's the tree obtained by ''subordinating $T$ by the local time'' in the sense of  \cite{Subor}, see also \cite{BertoinLeGallLeJean}; in particular, by \cite[Proposition 5]{Subor} the metric space $(\tilde{\mathcal L}, \tilde d_L)$ is an  $\mathbb{R}$-tree.

The main result of this section is that $(\tilde{\mathcal L}, \tilde d_L)$ can be seen as a continuous time branching tree
in the following sense.
Recall first from \cite[Chapter III]{AN} that a (one dimensional) continuous (time Markov) branching process
describes  a population where distinct individuals evolve   in continuous time, independently, one from the other and according to the same law.
Its dynamics are determined by a sequence $(\beta_k)_{k\geq 0}$ of nonnegative coefficients  with $\sum_{k\geq 0} \beta_k<\infty$ and called the branching rates. Specifically,  $\beta_k$ is the rate at which a typical individual  is replaced by 
$k$ new individuals. We stress that for $k=1$, the branching rate $\beta_1$ is essentially irrelevant and can be changed arbitrarily, since the replacement of an individual by another has no significant  effects on the model. A continuous time branching process starting from a single ancestor can thus be encoded from a Galton--Watson tree with reproduction law $(\beta_j/(\sum_{k\geq 0} \beta_k): j\geq 0)$ by assigning an independent  exponential length with mean $\sum_{k\geq 0} \beta_k$ to each edge. 
One calls a continuous time branching process subcritical when
\begin{equation} \label{E:sccbp}
 \sum_{k\geq 0}(k-1)\beta_k<0;
\end{equation}
this is the case if and only if the total length of its associated tree has finite expectation. 

It is convenient to introduce, for a generic marked decorated tree $\mathtt T\SB$, the notation  $n(\mathtt T\SB)$ for the number of points different from the root at which the decoration equals $1$ for the first time. So there is the identity
\begin{equation}\label{E:N&Z} 
n(\mathtt T\SB)= \left\{ \begin{matrix} Z(\mathtt T\SB) & 
\text{ if }g(\uprho\SB)\neq 1, \\
Z(\mathtt T\SB)+1 &\text{ if }g(\uprho\SB)=1. \end{matrix} \right.
\end{equation}

\begin{theorem}\label{T:Lcbt} 
Under\: $\P$, $(\tilde{\mathcal L}, \tilde d_L)$ is a sub-critical continuous branching tree with branching rates $(\beta_k)_{k\neq 1}$  given  by 
$$\beta_k= \mathbf{N}\SB\left( n(\mathtt T\SB)=k \right), \qquad \text{for any $k\neq 1$.}$$
\end{theorem}

Let us first explain quickly and  informally the two key steps of the proof. We start by observing that $(\tilde{\mathcal L}, \tilde d_L)$ is a sub-critical Crump--Mode--Jagers branching tree. Then taking a slightly different perspective makes us realize that the latter is actually a continuous branching tree.

Recall that  a Crump--Mode--Jagers (in short, C--M--J) branching process, also called a general branching process,  is a model of a population  in continuous time where distinct individuals evolve independently one from the other and according to the same law; see \cite[Section 3.3]{HJV}). Its distribution 
is entirely determined by the a.s. finite lifetime $l$ of a typical individual 
and a point process $\uptheta$ on $[0, l]$, called the reproduction point process, which records the age of the parent and the number of children that this individual begets alongs its life. Note that we allow $\uptheta$ to have an atom at time $l$ (in which case the individual begets children at its death-time). One generally agrees that $\uptheta$ has no atom at time $0$, meaning that an individual cannot beget children at its own birth-time (otherwise
could simply consider such children as sibling of this individual). 
In the subcritical case when $\uptheta([0,l])$ has expectation strictly less than $1$, 
a C--M--J branching process started from a single ancestor can thus be encoded as a compact random tree with a finite combinatorial structure.
\par
In order to simplify the statements, under $\P$, we consider an auxiliary homogeneous Poisson measure on $\R_+\times \T\SB$  with intensity $ \dd t \times \mathbf N\SB$, that we denote by $\upPhi$. In particular, recall from  Corollary \ref{C:excancestral}, that 
$$\sum_{a>0} \delta_{(a,R(\mathtt{T}\svn^{(a\bullet)}))}, $$
is distributed as $\upPhi$ stopped at the first atom $(t, T^\bullet)$ for which the decoration component evaluated at its marked point differs from $1$.

\begin{lemma}\label{L:LCMJt} 
Under $\P$,  $(\tilde{\mathcal L}, \tilde d_L)$ has the same law as a subcritical 
 {\normalfont{C--M--J}} branching tree such that the lifetime and reproduction point process of  a typical individual 
are  distributed as $l$ and $ \uptheta$, where:
\begin{itemize}
\item $l$ is the first atom in $\upPhi$ for which the decoration component evaluated at its marked point differs from $1$;
\item $\uptheta$ is the point process on $[0,l]$ such that $\uptheta$ has an atom  of multiplicity $k\geq 1$ at $t\in[0, l]$ if  and only if 
$\upPhi$ has an atom $(t,\mathtt{T}^{\bullet})$ with 
$Z(\mathtt{T}^{\bullet})=k$.
\end{itemize}
\end{lemma}

\begin{proof} Recall from Corollary \ref{C:excancestral} that the restriction of $\upPhi$ to $[0,l]\times \T\SB$ has the same law as the  
process of the excursions met on the ancestral segment. Observe that $l$ is a stopping time in the natural filtration of $\upPhi$ and has the exponential distribution with mean $\E(L\svn(z\svn))=v(0)$. In particular, Corollary \ref{C:Zint} entails that 
$$\E(\uptheta([0,l])) =  v(0) \int_{\T\SB} Z(\mathtt{T}\SB) \mathbf N\SB(\dd \mathtt{T}\SB)<1.$$

Next recall the impact of excursions on the structure of $(\tilde{\mathcal L}, \tilde d_L)$ that has been analyzed after Corollary \ref{C:Zint}. 
This leads us to interpret $\uptheta$ as a reproduction process induced by the excursions met on the ancestral segment.
We complete the proof with an appeal to the local decomposition 
stated in Lemma \ref{L:locdecsvn}, from which the branching property of C--M--J processes can be deduced. 
\end{proof}

We can now establish Theorem \ref{T:Lcbt}.

\begin{proof}[Proof of Theorem \ref{T:Lcbt}] 
We start making the following general observation. 
Consider a subcritical C--M--J branching process for which individuals may have children at the very time when they die, that is 
such that the reproduction point process $\uptheta$ may have an atom at the lifetime $l$. Recall also that on the other hand, individuals never beget children at their own birth-time, so $\uptheta(\{0\})=0$ a.s. Now if we imagine that each time an individual dies and simultaneously begets children, we pick one of those children uniformly at random and view it as the reincarnation (i.e. prolongation) of its parent, then we get another C--M--J branching process that generates the same compact random tree, but where typical individuals have a longer lifetime $\hat l$ and a modified reproduction point process $\hat \uptheta$. More precisely, $\hat l$ and  $\hat \uptheta$ can be described as follows.

Introduce a sequence $(l_n,\uptheta_n)_{n\geq 1}$ of i.i.d. copies of $(l,\uptheta)$ and let $N$ be the first index $n$ for which $\uptheta_n$ has no atom at $l_n$; in particular $N$ has a geometric law. For any $n<N$, write  $\uptheta'_n= \uptheta_n - \delta_{l_n}$ (i.e. we remove one of the children born at the death-time). Then $\hat l = l_1+ \cdots + l_N$ and $\hat \uptheta$ is the concatenation of 
the point processes $\uptheta'_n$ for $1\leq n<N$ and of $\uptheta_N$. 
 
By Lemma \ref{L:LCMJt}, we get from elementary concatenation properties of homogeneous Poisson random measures,  that 
  under $\P$, $(\tilde{\mathcal L}, \tilde d_L)$ has the same law as a subcritical 
 {\normalfont{C--M--J}} branching tree such that the lifetime and reproduction point process of  a typical individual 
are (jointly) distributed as $\hat l$ and $\hat \uptheta$, where $\hat l$  is the first time at which $\upPhi$ has an atom $(t,\mathtt{T}^{\bullet})$ with 
$n(\mathtt{T}^{\bullet})=0$ and $\hat \uptheta$ is the restriction of $\uptheta'$ to  $[0,\hat l]$. Here we used that for atoms $(t, \mathtt{T}^{\bullet})$ of $\Phi$, with decoration at the mark point different from $1$, it holds $n(\mathtt{T}^{\bullet})=Z(\mathtt{T}^{\bullet})$. 

We finally observe that the construction of $\hat l$ and $\hat \uptheta$ can be made simpler by an application of the mapping theorem for Poisson measures.
Introduce a family $\upvarphi_k$ for every  $k\neq 1$ of independent homogeneous Poisson point process on $\R_+$ with intensity $\beta_k$ defined in the statement. 
Then we can define $\hat l$ as the first point of $\upvarphi_0$ and $\hat \uptheta$ as the point measure on $[0, \hat l]$ that has an atom at $t$ with multiplicity $k\geq 1$ if and only if 
$\upvarphi_k$ has an atom at time $t$. It is now easy to verify that a  C--M--J branching process 
where a typical individual has lifetime $\hat l$ and reproduction process $\hat \uptheta$ is actually a continuous branching process with subcritical branching rates $(\beta_k)_{k\neq 1}$. 
\end{proof} 

\subsection{Excursions as a Poisson point process} \label{S:ExPPP}
Our final purpose here is to study general excursions in the terminology introduced at the beginning of this section.
Recall from Corollary \ref{C:Lnull}(iv) that the debut of an excursion belongs to a unique individual level set $\mathcal L_u$ for some $u\in \mathbb L$.
It is  further convenient to systematically mark excursion-subtrees at their end point. 
We also point out that for every $u\neq \uprho$, the decoration at the root $\uprho_u$ of the segment $S_u$ differs from $1$, and 
if we write $h_u=\inf\{t\in S_u: f_u(t)=1\}$ for the first hitting time of $1$ on $S_u$
(with our unusual convention that $h_u=z_u$ if $f_u$ avoids $1$ on $S_u$), then the segment $\llbracket \uprho_u, h_u\rrbracket$ is 
part of an excursion subtree whose debut belongs  $\mathcal{L}_v$ for the parent $v$ of $u$ in $\mathbb {L}$.

We first note that the intensity of (marked) excursions can easily be computed.

\begin{lemma} \label{L:intexc} For every measurable set $A\subset \T\SB$, the mean number of excursions of the ssMt in $A$ is
$$\textrm w^{(\gamma_0)}(0)\cdot  \mathbf N\SB(A).$$

\end{lemma} 

\begin{proof} Recall  the  definition \ref{E:defQ} and consider any $u\in \U$. An application of the Markov property at the first passage time at level $1$ of the decoration on a segment $S_u$ shows that the conditional expectation given $u\in \mathbb L$ of  number of excursions in $A$ and debut in $S_u$ equals 
$v(0) \mathbf N\SB(A)$. Since by \eqref{E:Lfinite},  the mean of $\#\mathbb L$ is $\mathrm{w}^{(\gamma_0)}(0)/v(0)$, we get the formula of the statement by summation. 
\end{proof}

Using an  obvious notation that was introduced previously for the ancestor $\varnothing$  only, each and every excursion of the ssMt thus appears
as an excursion of the type $R(\mathtt{T}_u^{(a\bullet)})$ for a unique individual $u\in \mathbb L$ and unique $0<a\leq L_u(z_u)$ with
$L_u^{-1}(a-)<L_u^{-1}(a)$.

We now enumerate the individuals in $\mathbb{L}$, say $u(1),\; u(2),\; \ldots,\; u(\#\mathbb{L}),$ and concatenate the excursion processes relative to each of these individuals, in the sense of Corollary~\ref{C:excancestral}.

There are many natural choices for such an indexing. In order to simplify some arguments, we adopt a depth-first order with respect to the local time. More precisely, we define by induction a sequence $(u(i), \mathbb{L}(i))_{1 \leq i \leq \#\mathbb{L}}$ as follows. Set $u(1) := \varnothing$, the root of $\mathbb{L}$, and let $\mathbb{L}(1)$ be the set of all direct descendants of $u(1)$ in $\mathbb{L}$, that is, the points $v \in \mathbb{L}$ such that there is no $v' \in \mathbb{L} \setminus \{\varnothing,v\}$ with $v' \preceq v$. Assume that $(u(i), \mathbb{L}(i))$ has been defined. If $\mathbb{L}(i)$ is empty, the procedure stops. Otherwise, we choose $u(i+1)$ as the smallest element (with respect to the lexicographical order) among those $u \in \mathbb{L}(i)$ that maximize
$$
L\big(1, \llbracket \uprho, \uprho(u) \rrbracket\big),
$$
where $\uprho(u)$ denotes the point of $T$ corresponding to the origin of $S_u$. We then define $\mathbb{L}(i+1)$ as the union of $\mathbb{L}(i) \setminus \{u(i+1)\}$ with the set of all direct descendants of $u(i+1)$ in $\mathbb{L}$. Observe that when the algorithm terminates, all elements of $\mathbb{L}$ have been discovered. 

To define the concatenation of excursions, we first introduce
$$
A(j) := \sum_{i=1}^{j} L_{u(i)}(z_{u(i)}), \qquad j = 0, \ldots, \#\mathbb{L}.
$$
In particular, $A(\#\mathbb{L}) = L(1,T)$ is the total local time at level~$1$, while $A(0)=0$. We then define the concatenated excursion process  $\mathcal{E}$ as the point process on $[0, L(1,T)]\times  \T\SB$ given by
$$\mathcal{E} = \sum_{j=1}^{\#\mathbb{L}} \sum_{a>0} \delta_{(A(j-1)+a, R(\mathtt{T}_{u(j)}^{(a\bullet)}))}.$$
Roughly speaking, $\mathcal{E}$ is a natural (though not canonical) way of  ordering the excursions of the ssMt by a random countable subset of $\R_+$.

\par

Our main goal here is to show that $\mathcal{E}$ has simple description in terms of the remarkable Poisson random measure  $\upPhi$    introduced before 
Lemma \ref{L:LCMJt} ; we recall that $\upPhi$ is a homogeneous Poisson measure on $\R_+\times \T\SB$  with intensity $ \dd t \times \mathbf Q\SB$. We  also consider the integer-valued additive functional
\begin{equation}\label{eq:def:F:expl}
F(t) = \int_{[0,t]\times \T\SB} \left( n(\mathtt{T}\SB)-1\right) \upPhi(\dd s, \dd \mathtt{T}\SB), \quad t\geq 0,
\end{equation}
where we recall that $n(\mathtt T\SB)$ denotes the number of points different from the root at which the decoration equals $1$ for the first time.
In particular $n(\mathtt{T}\SB)-1=-1$ if and only if  the decoration avoids $1$ aside at the root, and $n(\mathtt{T}\SB)-1=0$ when the root and the marked point are the only two locations on $\mathtt{T}\SB$  at which the decoration is $1$.

\begin{theorem}\label{T:EPPP} Under $\P$, the excursion process $\mathcal{E}$ has the same law as the homogeneous Poisson measure $\upPhi$
restricted to $[0,\varphi]\times \T\SB$, where $\varphi=\inf\{t\geq 0: F(t)=-1\}$. 
 \end{theorem}
 \begin{proof} 
 Recall from Corollary~\ref{C:excancestral} that, if $l$ denotes the time of the first atom of $\upPhi$ whose decoration component, evaluated at its marked point, differs from~$1$, then the restriction of $\upPhi$ to $[0,l]\times \T\SB$ has the same law as the point process of excursions encountered along the ancestral segment. In this setting, it follows from~\eqref{E:N&Z} and from the observation made just before Corollary~\ref{C:Zint} that $F(l)$ can be interpreted as the number of children of the ancestor in $\mathbb{L}$.

We then continue the exploration with the excursions encountered along the segments corresponding to the individuals of the first generation in $\mathbb{L}$. We begin with the individual $u \in \mathbb{L}$ of the first generation that maximizes $L\big(1, \llbracket \uprho, \uprho(u) \rrbracket\big)$, together with its progeny; in the case of a tie, we select the individual that is smallest with respect to the lexicographical order on $\mathbb{L}$.  This is precisely where the indexing $u(1), \ldots, u(\#\mathbb{L})$ plays a role. The local decomposition of Lemma~\ref{L:locdecsvn} implies that, conditionally on the size of this first generation, say $k$, the $k$ excursion processes associated with the first-generation individuals are independent and identically distributed, each having the same law as the point process of excursions encountered along the ancestral segment.

We proceed iteratively, exploring always in the depth-first order with respect to the local time, until all individuals in $\mathbb{L}$ have been explored. By the same concatenation property of homogeneous Poisson random measures that was already used in the proof of Theorem~\ref{T:Lcbt}, this allows us to view the entire excursion process $\mathcal{E}$ as the restriction of the homogeneous Poisson measure $\upPhi$ to $[0,\varphi]\times \T\SB$, for a suitable stopping time~$\varphi$. The explicit formula for $\varphi$ given in the statement is standard in the context of exploring a discrete planar tree by a walk.
 \end{proof}

We now apply Theorem~\ref{T:EPPP} to describe the conditional distribution of $\mathcal{E}$ given $(\tilde{\mathcal L}, \tilde d_L)$. To this end, recall first that $L(1,T)$ coincides with the total length of $\tilde{\mathcal{L}}$.  Consider an indexing $(\tilde{s}_i)_{1 \leq i \leq M}$ of the points of $\tilde{\mathcal{L}}$ that are distinct from the root and whose multiplicity is different from~$2$, and assume that this indexing is measurable with respect to $(\tilde{\mathcal L}, \tilde d_L)$. Recall that the multiplicity $m(\:\tilde{s}\:)$ of a point $\tilde{s} \in \tilde{\mathcal{L}}$ is defined as the number of connected components of $\tilde{\mathcal{L}} \setminus \big\{\tilde{s}\big\}$.  From the discussion following Corollary \ref{C:Zint}, to each $\tilde{s}_i$ corresponds a unique $R(\mathtt{T}_u^{(a_\bullet)})$ for which $n(\mathtt{T}_u^{(a_\bullet)}) \neq 1$, and for simplicity we denote it by $e_i$. In particular, for every $1 \leq i \leq M$, the root of $e_i$ corresponds in $\tilde{\mathcal{L}}$ to the point $\tilde{s}_i$, and we have $n(e_i) = m(\tilde{s}_i) - 1.$

\begin{corollary}\label{cor:final} Under $\P$,  conditionally on  $(\tilde{\mathcal L}, \tilde d_L)$:
 \begin{enumerate}
        \item The excursions $(e_i)_{1\leq i\leq M}$  are independent, and with respective distributions 
        \begin{equation*}
            \mathbf N\SB\big( \,  \cdot \,  \big|\: n(\mathtt{T}^\bullet) = {{m}}(\tilde{s}_i)-1\big),   \qquad  \text{for }  1\leq i\leq M. 
        \end{equation*}
        \item The point measure $\mathcal{E}(\mathrm{d}t \, \mathrm{d} \mathtt{T}^{\bullet})\mathbf{1}_{n(\mathtt{T}^{\bullet}) = 1}$ is independent of $\mathcal{E}(\mathrm{d}t \, \mathrm{d} \mathtt{T}^{\bullet})\mathbf{1}_{n(\mathtt{T}^{\bullet}) \neq 1}$ , and distributed as a Poisson point measure with intensity $$\mathbf{1}_{[0,L(1,T)]}(t)\mathrm{d} t \times \mathbf N\SB( \, \cdot \cap \{n(\mathtt{T}^\bullet) =1\}).$$ 
\end{enumerate}
 \end{corollary}

\begin{proof}
Let  $\tilde{\mathcal{E}}$ be the image  measure of $\mathcal{E}\mathbf{1}_{\{n(\mathtt{T}^{\bullet}) \neq  1\}}$ by the map $(t, \mathtt{T}^\bullet)\mapsto \big(t, n(\mathtt{T}^\bullet)\big)$. Remark that the number of atoms of $\tilde{\mathcal{E}}$ is precisely $M$ and  let  $(t_i,k_i)_{1\leq i\leq M}$ be the list of these atoms  indexed in increasing  order with respect to their first coordinate, i.e. $t_1<t_2<\dots< t_M$. By Theorem~\ref{T:EPPP}, conditionally on $\tilde{\mathcal{E}}$:
\begin{itemize}
\item 
The point measure $\mathcal{E}\mathbf{1}_{\{n(\mathtt{T}^{\bullet}) \neq  1\}}$
is a marked  point measure.
More precisely, conditionally on $\tilde{\mathcal{E}}$, let
$(\mathtt T_i^\bullet)_{1\leq i\leq M}$ be independent marked decorated trees such that
$\mathtt T_i^\bullet$ has distribution
$\mathbf N \SB(\,\cdot\,|\, n(\mathtt{T}^\bullet)=k_i)$.
Then, conditionally on $\tilde{\mathcal{E}}$, the point measure
$\mathcal{E}$ has the same law as
$\sum_{1\leq i\leq M} \delta_{(t_i,\mathtt T_i^\bullet)}$.

\item The point measure
$\mathcal{E}(\mathrm{d}t\,\mathrm{d}\mathtt{T}^\bullet)
\mathbf{1}_{\{n(\mathtt{T}^{\bullet}) = 1\}}$
is independent of $\mathcal{E}(\mathrm{d}t\,\mathrm{d}\mathtt{T}^\bullet)
\mathbf{1}_{\{n(\mathtt{T}^{\bullet}) \neq 1\}}$ and of
$\tilde{\mathcal{E}}$, and is distributed as a Poisson point measure with
intensity
$$\mathbf{1}_{[0,L(1,T)]}(t)\,\mathrm{d}t\times
\mathbf N \SB (\,\cdot\,\cap \{n(\mathtt{T}^\bullet)=1\}).$$
\end{itemize}
Since $L(1,T)$ is the total length of $\tilde{\mathcal{L}}$ and
$n(e_i)=m(\tilde{s}_i)-1$ for $1\leq i\leq M$, 
it only remains to show that $(\tilde{\mathcal L}, \tilde d_L)$ is a functional of
$\tilde{\mathcal{E}}$. Let us now explain how to construct $\tilde{\mathcal{L}}$ inductively  using  $\widetilde{\mathcal{E}}$. We start with a root segment of length $t_1$ that we consider as active.  Next depending on the value of $k_1$ we proceed as follows:
\begin{itemize}
\item if $k_1=0$, the segment ends without generating further children; in particular, the extremity of the segment of length $t_1$ is a leaf;
\item if $k_1\ge 2$, we create a branching point at the extremity of the segment of length $t_1$, and  $k_1$ new active segments are grafted. 
\end{itemize} In both cases, the initial segment becomes inactive. 
The construction continues by exploring the next active segment in
depth--first order and by recursively iterating this procedure until there is no active segments 
left. Namely, the algorithm stops after processing the atom with
entry $t_M$, and it always explores the active tip that is last in the
depth--first ordering. Clearly,   this procedure produces a finite compact
$\mathbb{R}$--tree. By construction, the number of branching points correspond to   the cardinality of $\{1\leq  i\leq M : k_i \geq 2 \}$, the multiplicity of  branching points are given by the entries $k_i + 1$ with $k_i \geq 2$, and the number of (non-root) leaves is given by the cardinality of $\{ 1\leq i \leq M : k_i = 0 \}$. By the discussion following Corollary~\ref{C:Zint} and the definition of the indexation $u(1),\dots, u(\#\mathbb{L})$,  the tree we just constructed is 
precisely $(\tilde{\mathcal{L}},\tilde d_L)$.  We leave these  verifications to
the reader, as they are straightforward and closely parallel the
C--M--J description of Section~\ref{S:Bsll}. This concludes the proof of the corollary.
\end{proof}

We conclude this section by briefly discussing the relationship between the present results and those obtained in
\cite{RRO1,RRO2}. In the later works, the object of study is a continuous Markov process, 
indexed by a (fractal)  Lévy tree $\tau$ - the family of random trees which arise as the scaling limit of (sub-)critical Galton--Watson trees. Informally, such a process $(Y_t:t\in\tau)$ can be described as follows: conditionally on the
underlying random tree $\tau$, the process evolves as a Markov process along each branch and splits into independent copies at every branching point.  In order to guarantee the existence of a non-trivial local time,  the existence of a fixed point $x$ of the state space which is  regular and instantaneous for  the Markov process is further assumed.
\par 
A central objective of \cite{RRO1} is the construction of a measure 
$A(x,\mathrm dt)$ supported on the level set $\{t\in\tau:Y_t=x\}$. In contrast with the present work, this construction
is technical, and the resulting measure is supported on the set of leaves of the tree - instead of the skeleton. The law of the motion along
the segment connecting the root of $\tau$ to a point sampled according to $A(x,\mathrm dt)$ is also identified, and exhibits
similarities with the spinal descriptions obtained in Section~\ref{sec:decora:spine}. The level set $\{t\in\tau:Y_t=x\}$ thus plays a role analogous to that of $\mathcal L$ in the present setting. As in this
work, its structure is analyzed by introducing a tree $\widetilde{\tau}$, which is obtained by identifying all pairs of
points $s,t\in\tau$ such that $A(x,\llbracket s,t\rrbracket)=0$. It is shown that $\widetilde{\tau}$ is itself a
(fractal) Lévy tree, whereas in the present setting $(\widetilde{\mathcal L},\widetilde d_L)$ is an essentially discrete Lévy tree. In both frameworks, the resulting tree satisfies the branching property in the sense of
\cite{Weill}.

The second work \cite{RRO2} is devoted to the study of excursions away from   $x$. It is shown
that, when properly indexed using the measure $A(x,\mathrm dt)$, the associated point process of excursions
has a Poissonian structure, yielding an analogue of the classical It\^o excursion theory.
Moreover, the conditional law of the excursions given $\widetilde{\tau}$ is identified by methods that are closely
related in spirit to those employed in the present work.

In particular, the role played here by the measure $\widetilde{\mathcal E}$ is taken over in \cite{RRO2} by a measure
$\widetilde{\mathcal M}$, which records the boundary sizes of excursion components. These quantities play a role
analogous to $n(\mathtt T^\bullet)$ in the present framework, and $\widetilde{\mathcal M}$ is obtained as a
push-forward of the excursion measure through a suitable projection. Defining the notion of excursion boundary size
is delicate, and the reader is referred to \cite{RRO2} for details. Finally, the functional \eqref{eq:def:F:expl} admits a natural analogue in \cite{RRO2} in the form of a
Lévy process obtained from $\widetilde{\mathcal M}$ via the Lévy--It\^o decomposition. 
Namely, it consists of the Lévy process which encodes  the Lévy tree $\widetilde{\tau}$.   Lastly, it is worth mentioning that the  key role played by local decompositions here is taken over in  \cite{RRO1,RRO2} by the so-called special Markov property.

\newpage 

\section*{Index of main notation}
\noindent \textbf{Lévy processes and self–similar Markov processes}
\vspace{-0.5cm}
\begin{center}
\begin{tabular}{p{0.25\textwidth}p{0.7\textwidth}}
$\xi=(\xi_t)_{t\ge0}$  &Lévy process (possibly killed)\\
$\zeta$             & Lifetime of $\xi$ (killing time, possibly $\infty$)\\
$P_x$               & Law of $\xi$ started from $x$; $P=P_0$\\
$\psi$              & Lévy exponent of $\xi$: $E(\e^{\gamma\xi_t}\ind{t<\zeta})=\e^{t\psi(\gamma)}$\\
$\Lambda$           & Lévy measure of $\xi$ on $[-\infty,\infty)$ (mass at $-\infty$ encodes killing rate)\\
$\sigma^2,\mathrm a$& Gaussian coefficient and drift of $\xi$\\
$\ell(y,t)$         & Local time at level $y$ for $\xi$ up to time $t$\\
$\ell(y,\dd t)$     & Corresponding local time Stieltjes measure on $\R_+$\\
$H_y$               & First hitting time of level $y$ by $\xi$\\
$K_y$               & Last passage time at level $y$\\
$v(y)$              & Potential density: $v(y)=E(\ell(y,\infty))$\\
$X=(X_t)_{0\le t<z}$ & Positive self–similar Markov process obtained from $\xi$ via Lamperti transformation\\
$z$                 & Lifetime of $X$: $z=\int_0^\zeta \exp(\alpha\xi_s)\,\dd s$\\
\end{tabular}

\end{center}

\noindent \textbf{Characteristic quadruplet and cumulant}
\vspace{-0.5cm}
\begin{center}
\begin{tabular}{p{0.25\textwidth}p{0.7\textwidth}}
$(\sigma^2,\mathrm a,\boldsymbol{\Lambda};\alpha)$
                    & Characteristic quadruplet defining a ssMt\\
$\boldsymbol{\Lambda}$ & Generalized Lévy measure on $[-\infty,\infty)\times\calS_1$\\
$\calS_1$           & Space of non–increasing sequences $\mathbf y=(y_n)_{n\ge1}$ with $y_n\in[-\infty,\infty)$ and $y_n\to-\infty$\\
$\boldsymbol{\Lambda}_0$ & Projection of $\boldsymbol{\Lambda}$ on $[-\infty,\infty)$ (Lévy measure of $\xi$)\\
$\boldsymbol{\Lambda}_1$ & Projection of $\boldsymbol{\Lambda}$ on $\calS_1$\\
$\kappa$            & Cumulant associated with $(\sigma^2,\mathrm a,\boldsymbol{\Lambda};\alpha)$\\
$\gamma_0>0$        & Parameter such that $\kappa(\gamma_0)<0$ (subcriticality assumption)\\
$\kappa^{(\gamma)}$ & Shifted exponent: $\kappa^{(\gamma)}(\cdot)=\kappa(\gamma+\cdot)$\\
$\mathrm{w}^{(\gamma)}$ & Potential density associated with $\kappa^{(\gamma)}$\\
\end{tabular}
\end{center}
\noindent \textbf{Conditioning on terminal value and tilting}
\vspace{-0.5cm}
\begin{center}
\begin{tabular}{p{0.25\textwidth}p{0.7\textwidth}}
$P_{x,y}$           & Law of $\xi$ started at $x$ and conditioned to have terminal value $y$ (killed at $K_y$)\\
$P^{(\beta)}$       & $\beta$–Esscher transform of $P$: Lévy exponent $\psi^{(\beta)}(\cdot)=\psi(\beta+\cdot)$\\
$P_x^{(\beta)}$     & Law of the tilted process started from $x$\\
$v^{(\beta)}(x)$    & Potential density under $P^{(\beta)}$: $v^{(\beta)}(x)=\e^{\beta x}v(x)$\\
$\hat P_x$          & Law of the dual process $-\xi$ started from $x$\\
$Q_{1,x}$           & Law of the pssMp with exponent $\kappa^{(\gamma_0)}$ started at $1$ and conditioned to end at $x$\\
$\hat Q_x$          & Law of the dual pssMp with exponent $\hat\kappa^{(\gamma_0)}(\cdot)=\kappa(\gamma_0-\cdot)$ started at $x$\\
$\hat Q_{x,1}$      & Law of the dual pssMp started at $x$ and conditioned to end at $1$\\
\end{tabular}
\end{center}

\newpage

\noindent \textbf{Trees and genealogy}
\vspace{-0.5cm}
\begin{center}
\begin{tabular}{p{0.25\textwidth}p{0.7\textwidth}}
$(T, d_{T},\uprho)$                 & Rooted compact real tree\\
$\uprho$            & Root of $T$\\
$\U$                & Ulam tree $\bigcup_{n\geq0}\N^n$ indexing individuals in the population\\
$u\in\U$            & Label of an individual \\
$S_u$               & Oriented segment of $T$ associated with individual $u$.\\
$|S_u|$             & Length of $S_u$, i.e.\ lifetime of individual $u$.\\
$\llbracket s,t\rrbracket$ & Geodesic segment in $T$ between $s$ and $t$.\\
$s\prec t$          & $s$ belongs to the open segment $\llbracket\uprho,t\llbracket$ (ancestor of $t$).\\[0.5em]
\end{tabular}
\end{center}

\noindent \textbf{Self–similar Markov trees (ssMt)}
\vspace{-0.5cm}
\begin{center}
\begin{tabular}{p{0.25\textwidth}p{0.7\textwidth}}
ssMt                & Self–similar Markov tree\\
$(T,g)$             &  real tree $T$ with decoration $g$\\
$f_u$               & Decoration along the segment $S_u$\\
$\eta_u$            & Reproduction point process on $S_u\times(0,\infty)$ for individual $u$\\
$\chi(u)$           & Decoration at the origin of $S_u$ (initial size of individual $u$)\\
$\P$                & Law of the whole family $(f_u,\eta_u)_{u\in\U}$ starting from  decoration $1$\\
$\Q$                & Law of the decorated tree $(T,g)$ (image of $\P$ on $\T$) starting from  decoration $1$\\
$\P_x,\Q_x$         & Same laws when the initial decoration at the root is $g(\uprho)=x>0$
\end{tabular}
\end{center}

\noindent \textbf{Measures and local times on ssMt} 
\vspace{-0.5cm}
\begin{center}
\begin{tabular}{p{0.25\textwidth}p{0.7\textwidth}}
$\uplambda$         & Length measure on $T$\\
$\uplambda^\gamma$  & Weighted length measure: $\uplambda^\gamma(\dd t)=g(t)^{\gamma-\alpha}\uplambda(\dd t)$\\
$\upmu$             & Harmonic measure on $T$\\
$\mathcal L(x)$     & Level set at level $x>0$: $\mathcal L(x)=\{t\in T: g(t)=x\}$ \\
$L(x,\dd t)$        & Local time measure at level $x>0$, supported on $\mathcal L(x)$\\
$L(x,t)$            & Repartition function of $L(x,\cdot)$ along an oriented segment\\
$L(x,T)$            & Total local time at level $x$ on the whole tree $T$\\
$\mathcal H_x$      & First hitting line  of the decoration at level  $x$\\
$N(x,\dd t)$        & Counting measure on $\mathcal H_x$; $N(x,T)=\#\mathcal H_x$\\
\end{tabular}
\end{center}

\noindent \textbf{Excursions} 
\vspace{-0.5cm}
\begin{center}
\begin{tabular}{p{0.25\textwidth}p{0.7\textwidth}}
$\mathcal L$ 
& Level set at level $1$: $\mathcal L=\{t\in T:g(t)=1\}$\\
$\mathcal L_u$ 
& Level set of the decoration on the segment $S_u$, for $u\in \mathbb{U}$\\
$\mathbb L$ 
& Subset of $\mathbb{U}$ of all individuals whose level set is non-empty\\
$d',d''$ 
& D\'ebut and end of an excursion-subtree\\
$R(\mathtt T^\bullet)$ 
& Reduction operator (pruning at first returns to level $1$)\\
$\mathbf N^\bullet$ 
& Excursion measure on $\T^\bullet$\\
$n(\mathtt T^\bullet)$ 
& Number of first returns to level $1$ away from the root\\
$\tilde{\mathcal L}$ 
& Quotient of $\mathcal L$ by the local-time pseudo-distance $d_L$\\
$\mathcal E$ 
& Concatenated excursion process\\
\end{tabular}
\end{center}

\newpage

\bibliography{LTonssMt.bib}

@article{ALG15,
  author   = {Abraham, Christophe and Le Gall, Jean-Fran{\c{c}}ois},
  title    = {Excursion theory for {B}rownian motion indexed by the {B}rownian tree},
  journal  = {J. Eur. Math. Soc.},
  volume   = {20},
  pages    = {2951--3016},
  year     = {2018},
  language = {English}
}

@article{Subor,
  author   = {Le Gall, Jean-Fran{\c{c}}ois},
  title    = {Subordination of trees and the Brownian map},
  journal  = {Probab. Theory Relat. Fields},
  volume   = {171},
  pages    = {819--864},
  year     = {2018},
  language = {English}
}

@article{BertoinLeGallLeJean,
  author   = {Bertoin, Jean and Le Gall, Jean-Fran{\c{c}}ois and Le Jan, Yves},
  title    = {Spatial branching processes and subordination},
  journal  = {Canadian Journal of Mathematics},
  volume   = {49},
  pages    = {24--54},
  year     = {1997},
  language = {English}
}

@article{Weill,
  author    = {Weill, Matthieu},
  title     = {Regenerative real trees},
  journal   = {Ann. Probab.},
  fjournal  = {Ann. Probab.},
  volume    = {35},
  pages     = {2091--2121},
  year      = {2007},
  language  = {English}
}

@article{LeGallPerkins2025,
  author    = {Le Gall, Jean-François and Perkins, Edwin},
  title     = {A stochastic differential equation for local times of super-Brownian motion},
  journal   = {Annals of Probability},
  fjournal  = {Ann. Probab.},
  volume    = {53},
  pages     = {355--390},
  year      = {2025},
  language  = {English},
  doi       = {10.1214/24-AOP1707}
}

@article{HoS,
  author    = {van der Hofstad, R. and Slade, G.},
  title     = {Convergence of critical oriented percolation to super-Brownian motion above 4+1 dimensions},
  journal   = {Annales de l'Institut Henri Poincaré. Probabilités et Statistiques},
  fjournal  = {Ann. Inst. H. Poincaré Probab. Statist.},
  volume    = {39},
  pages     = {413--485},
  year      = {2003},
  language  = {English},
  doi       = {10.1016/S0246-0203(03)00011-3},
  zbMATH    = {02014639},
  Zbl       = {1030.60079}
}

@article{HaS,
  author    = {Hara, T. and Slade, G.},
  title     = {The scaling limit of the incipient infinite cluster in high-dimensional percolation. II. Integrated super-Brownian excursion},
  journal   = {Journal of Mathematical Physics},
  fjournal  = {J. Math. Phys.},
  volume    = {41},
  pages     = {1244--1293},
  year      = {2000},
  language  = {English},
  doi       = {10.1063/1.533220},
  zbMATH    = {01551544},
  Zbl       = {0953.60098}
}

@article{DS,
  author    = {Derbez, E. and Slade, G.},
  title     = {The scaling limit of lattice trees in high dimensions},
  journal   = {Communications in Mathematical Physics},
  fjournal  = {Comm. Math. Phys.},
  volume    = {193},
  pages     = {69--104},
  year      = {1998},
  language  = {English},
  doi       = {10.1007/s002200050320},
  zbMATH    = {01166056},
  Zbl       = {0911.60077}
}

@article{Fourati,
 author = {Fourati, Sonia},
 title = {Vervaat et {L{\'e}vy}},
 fjournal = {Annales de l'Institut Henri Poincar{\'e}. Probabilit{\'e}s et Statistiques},
 journal = {Ann. Inst. Henri Poincar{\'e}, Probab. Stat.},
 issn = {0246-0203},
 volume = {41},
 number = {3},
 pages = {461--478},
 year = {2005},
 language = {English},
 doi = {10.1016/j.anihpb.2004.11.001},
 url = {https://eudml.org/doc/77854},
 zbMATH = {2191863},
 Zbl = {1074.60082}
}

@book{LP,
 author = {Bertoin, Jean},
 title = {L{\'e}vy {P}rocesses},
 fseries = {Cambridge Tracts in Mathematics},
 series = {Camb. Tracts Math.},
 issn = {0950-6284},
 volume = {121},
 isbn = {0-521-64632-4},
 year = {1998},
 publisher = {Cambridge: Cambridge Univ. Press},
 language = {English},
 zbMATH = {1232408},
 Zbl = {0938.60005}
}

@book{AN,
 author = {Athreya, Krishna B. and Ney, Peter E.},
 title = {Branching {P}rocesses.},
 edition = {Reprint of the 1972 original},
 isbn = {0-486-43474-5},
 year = {2004},
 publisher = {Mineola, NY: Dover Publications},
 language = {English},
 zbMATH = {2119076},
 Zbl = {1070.60001}
}

@book{HJV,
 author = {Haccou, Patsy and Jagers, Peter and Vatutin, Vladimir A.},
 title = {Branching {P}rocesses. {Variation}, {G}rowth, and {E}xtinction of {P}opulations.},
 fseries = {Cambridge Studies in Adaptive Dynamics},
 series = {Camb. Stud. Adapt. Dyn.},
 volume = {5},
 isbn = {978-0-521-83220-5},
 year = {2005},
 publisher = {Cambridge: Cambridge University Press},
 language = {English},
 doi = {10.1017/CBO9780511629136},
 zbMATH = {5172066},
 Zbl = {1118.92001}
}

@book{RY,
 author = {Revuz, Daniel and Yor, Marc},
 title = {Continuous {M}artingales and {Brownian} {M}otion.},
 edition = {3rd ed.},
 fseries = {Grundlehren der Mathematischen Wissenschaften},
 series = {Grundlehren Math. Wiss.},
 issn = {0072-7830},
 volume = {293},
 isbn = {3-540-64325-7},
 year = {1999},
 publisher = {Berlin: Springer},
 language = {English},
 zbMATH = {1245556},
 Zbl = {0917.60006}
}

@article{Sugitani,
 author = {Sugitani, Sadao},
 title = {Some properties for the measure-valued branching diffusion processes},
 fjournal = {Journal of the Mathematical Society of Japan},
 journal = {J. Math. Soc. Japan},
 issn = {0025-5645},
 volume = {41},
 number = {3},
 pages = {437--462},
 year = {1989},
 language = {English},
 doi = {10.2969/jmsj/04130437},
 zbMATH = {4121156},
 Zbl = {0684.60049}
}

@article{Alili,
 author = {Alili, L. and Chaumont, L. and Graczyk, P. and {\.Z}ak, T.},
 title = {Space and time inversions of stochastic processes and {Kelvin} transform},
 fjournal = {Mathematische Nachrichten},
 journal = {Math. Nachr.},
 issn = {0025-584X},
 volume = {292},
 number = {2},
 pages = {252--272},
 year = {2019},
 language = {English},
 doi = {10.1002/mana.201700152},
 zbMATH = {7047877},
 Zbl = {1415.31004}
}

@article{Kypsurvey,
 author = {Kyprianou, Andreas E.},
 title = {Stable {L{\'e}vy} processes, self-similarity and the unit ball},
 fjournal = {ALEA. Latin American Journal of Probability and Mathematical Statistics},
 journal = {ALEA, Lat. Am. J. Probab. Math. Stat.},
 issn = {1980-0436},
 volume = {15},
 number = {1},
 pages = {617--690},
 year = {2018},
 language = {English},
 url = {alea.impa.br/articles/v15/15-25.pdf},
 zbMATH = {6876312},
 Zbl = {1396.60051}
}

@article{BM,
 author = {Bertoin, Jean and Mallein, Bastien},
 title = {Infinitely ramified point measures and branching {L{\'e}vy} processes},
 fjournal = {The Annals of Probability},
 journal = {Ann. Probab.},
 issn = {0091-1798},
 volume = {47},
 number = {3},
 pages = {1619--1652},
 year = {2019},
 language = {English},
 doi = {10.1214/18-AOP1292},
 zbMATH = {7067278},
 Zbl = {1456.60225}
}

@book{Kyp,
 author = {Kyprianou, Andreas E.},
 title = {Fluctuations of {L{\'e}vy} {P}rocesses with {A}pplications. {Introductory} lectures},
 edition = {2nd},
 fseries = {Universitext},
 series = {Universitext},
 issn = {0172-5939},
 isbn = {978-3-642-37631-3; 978-3-642-37632-0},
 year = {2014},
 publisher = {Berlin: Springer},
 language = {English},
 doi = {10.1007/978-3-642-37632-0},
 zbMATH = {6176054},
 Zbl = {1384.60003}
}

@book{KypPar,
 author = {Kyprianou, Andreas E. and Pardo, Juan Carlos},
 title = {Stable {L{\'e}vy} {P}rocesses via {Lamperti}-{T}ype {R}epresentations},
 fseries = {Institute of Mathematical Statistics Monographs},
 series = {Inst. Math. Stat. Monogr.},
 volume = {7},
 isbn = {978-1-108-48029-1; 978-1-108-64831-8},
 year = {2022},
 publisher = {Cambridge: Cambridge University Press},
 language = {English},
 doi = {10.1017/9781108648318},
 zbMATH = {7504322}
}

@book{Sato,
 author = {Sato, Ken-Iti},
 title = {L{\'e}vy {P}rocesses and {I}nfinitely {D}ivisible {D}istributions},
 fseries = {Cambridge Studies in Advanced Mathematics},
 series = {Camb. Stud. Adv. Math.},
 volume = {68},
 isbn = {0-521-55302-4},
 year = {1999},
 publisher = {Cambridge: Cambridge University Press},
 language = {English},
 zbMATH = {1402217},
 Zbl = {0973.60001}
}

@article{RRO1,
 author = {Riera, Armand and Rosales-Ortiz, Alejandro},
 title = {The structure of the local time of {Markov} processes indexed by {L{\'e}vy} trees},
 fjournal = {Probability Theory and Related Fields},
 journal = {Probab. Theory Relat. Fields},
 issn = {0178-8051},
 volume = {189},
 number = {1-2},
 pages = {1--99},
 year = {2024},
 language = {English},
 doi = {10.1007/s00440-023-01258-w},
 zbMATH = {7859940},
 Zbl = {1543.60093}
}

@misc{RRO2,
 author = {Riera, Armand and Rosales-Ortiz, Alejandro},
 title = {Excursion theory for {Markov} processes indexed by {Levy} trees},
 year = {2025+},
 howpublished = {Preprint, {arXiv}:2411.12717 [math.{PR}]},
 url = {https://arxiv.org/abs/2411.12717},
 arXiv = {arXiv:2411.12717}
}

@book{ssMt,
 author = {Bertoin, Jean and Curien, Nicolas and Riera, Armand},
 title = {Self-similar {Markov} {T}rees and {S}caling {L}imits},
 year = {2026+},
 publisher = {Cambridge University Press; to appear in IMS Monographs,  {arXiv}:2407.07888 [math.{PR}]},
 url = {https://arxiv.org/abs/2407.07888},
 arXiv = {arXiv:2407.07888}
}

@book{JS,
 author = {Jacod, Jean and Shiryaev, Albert N.},
 title = {Limit {T}heorems for {S}tochastic {P}rocesses.},
 edition = {2nd},
 fseries = {Grundlehren der Mathematischen Wissenschaften},
 series = {Grundlehren Math. Wiss.},
 issn = {0072-7830},
 volume = {288},
 isbn = {3-540-43932-3},
 year = {2003},
 publisher = {Berlin: Springer},
 language = {English},
 zbMATH = {1834045},
 Zbl = {1018.60002}
}

@article{BGLT,
 author = {Blumenthal, R. and Getoor, R.},
 title = {Local times for {Markov} processes},
 fjournal = {Zeitschrift f{\"u}r Wahrscheinlichkeitstheorie und Verwandte Gebiete},
 journal = {Z. Wahrscheinlichkeitstheor. Verw. Geb.},
 issn = {0044-3719},
 volume = {3},
 pages = {50--74},
 year = {1964},
 language = {English},
 doi = {10.1007/BF00531683},
 zbMATH = {3205726},
 Zbl = {0126.33701}
}

@article {DLG,
    AUTHOR = {Duquesne, Thomas and Le Gall, Jean-Fran\c cois},
     TITLE = {Random trees, {L}\'evy processes and spatial branching
              processes},
   JOURNAL = {Ast\'erisque},
  FJOURNAL = {Ast\'erisque},
    NUMBER = {281},
      YEAR = {2002},
     PAGES = {vi+147},
      ISSN = {0303-1179,2492-5926},
   MRCLASS = {60J80 (60F17 60G51)},
  MRNUMBER = {1954248},
MRREVIEWER = {David\ J.\ Aldous},
}

@article {MytPer,
    AUTHOR = {Mytnik, Leonid and Perkins, Edwin},
     TITLE = {Regularity and irregularity of {$(1+\beta)$}-stable
              super-{B}rownian motion},
   JOURNAL = {Ann. Probab.},
  FJOURNAL = {The Annals of Probability},
    VOLUME = {31},
      YEAR = {2003},
    NUMBER = {3},
     PAGES = {1413--1440},
      ISSN = {0091-1798,2168-894X},
   MRCLASS = {60G57 (60G17 60H15)},
  MRNUMBER = {1989438},
MRREVIEWER = {Ilya\ S.\ Molchanov},
       DOI = {10.1214/aop/1055425785},
       URL = {https://doi.org/10.1214/aop/1055425785},
}

\end{document}